\documentclass[11pt]{article}

\usepackage[T1]{fontenc}
\usepackage[english]{babel}
\usepackage[a4paper,margin=1in]{geometry}
\usepackage{amsmath,amssymb,amsthm,mathtools,bm}
\usepackage{booktabs,array}
\usepackage{microtype}
\usepackage{hyperref}
\usepackage{enumitem}
\usepackage{cleveref}
\allowdisplaybreaks
\hypersetup{
  hidelinks,
  pdftitle={Intersections of Directed Graphs},
  pdfauthor={Zhanping Yang and Qinghou Zeng},
  pdfsubject={Intersection discrepancy for weighted directed graphs},
  pdfkeywords={discrepancy, directed graphs, weighted graphs, transpositions, relabellings, Gram matrices}
}

\newtheorem{theorem}{Theorem}[section]
\newtheorem{lemma}[theorem]{Lemma}
\newtheorem{corollary}[theorem]{Corollary}
\newtheorem{proposition}[theorem]{Proposition}

\newtheorem{claim}[theorem]{Claim}
\newtheorem{problem}[theorem]{Problem}

\theoremstyle{remark}
\newtheorem{remark}[theorem]{Remark}

\newcommand{\E}{\mathbb{E}}
\newcommand{\disc}{\operatorname{disc}}
\newcommand{\R}{\mathbb{R}}
\newcommand{\1}{\mathbf{1}}

\newcommand{\rank}{\operatorname{rank}}

\newcommand{\ip}[2]{\left\langle #1,#2\right\rangle}
\newcommand{\norm}[1]{\left\lVert #1\right\rVert}
\newcommand{\abs}[1]{\left|#1\right|}

\title{Intersections of Directed Graphs}
\author{
  Zhanping Yang\thanks{Center for Discrete Mathematics, Fuzhou University, Fuzhou, Fujian 350108, China. Email: \texttt{yangzp@163.com}.}
  \and
  Qinghou Zeng\thanks{Corresponding author. Center for Discrete Mathematics, Fuzhou University, Fuzhou, Fujian 350108, China. Research supported by National Key R\&D Program of China (Grant No. 2023YFA1010202) and National Natural Science Foundation of China (Grant No. 12371342). Email: \texttt{zengqh@fzu.edu.cn}.}
}
\date{}

\begin{document}
\maketitle

\begin{abstract}
Given two weighted directed graphs of order \(n\), we study how much their overlap can deviate from its random average under relabelling, and how concentrated the distribution of their intersection is when they are placed at random. Bollob\'as and Scott studied these problems on  weighted $k$-uniform hypergraphs and raised the corresponding questions for directed graphs. In this paper, we study these problems for weighted directed graphs and obtain results analogous to those in the weighted hypergraph setting.
\end{abstract}

\noindent\textbf{2020 Mathematics Subject Classification.} 05C35, 05C69.

\noindent\textbf{Keywords.} discrepancy, directed graph, weighted graph, graph intersection

\section{Introduction}\label{sec:intro}
Graph discrepancy measures how uniformly the edges of a graph are
distributed among its vertex subsets. Let \(G\) be a graph with \(n\) vertices and edge
density \(p\), and write
\begin{equation*}
\disc_p^+(G)=\max_{S\subseteq V(G)}\left(e(G[S])-p\binom{|S|}{2}\right),
\end{equation*}
and
\begin{equation*}
\disc_p^-(G)=\max_{S\subseteq V(G)}\left(p\binom{|S|}{2}-e(G[S])\right),
\end{equation*}
and set $\disc_p(G)=\max\{\disc_p^+(G),\disc_p^-(G)\}$. Erd\H{o}s and Spencer \cite{ErdosSpencer1972} showed that every graph of order \(n\) and density \(1/2\) has discrepancy of order at least \(n^{3/2}\); more generally, the corresponding scale for \(k\)-uniform hypergraphs is \(n^{(k+1)/2}\). These bounds are sharp up to the constant factor, as shown by random examples. Erd\H{o}s, Goldberg, Pach and Spencer \cite{ErdosGoldbergPachSpencer1988} extended this result to graphs of arbitrary density. More precisely, if $2/(n-1)<p<1-2/(n-1)$, then $\disc_p(G)\ge c_p\,n^{3/2}$. Bollob\'as and Scott
\cite{BollobasScottDiscrepancy2006} proved that, if \(p(1-p)\ge 1/n\), then $\disc_p^+(G)\disc_p^-(G)\ge c\,p(1-p)n^3$ for an absolute constant \(c>0\). In particular, $\disc_p(G)\ge c'\sqrt{p(1-p)}\,n^{3/2}$. They also obtained analogous results for uniform hypergraphs. R\"aty, Sudakov and Tomon \cite{RatySudakovTomon} studied the positive discrepancy of a single graph in a wide range of densities. For general background on discrepancy theory, we refer to Beck and S\'os \cite{BeckSosDiscrepancy1995} and Chazelle \cite{CDiscrepancy2000}.

 Bollob\'as and Scott \cite{BollobasScottGraphs2011} subsequently
introduced the intersection discrepancy of two graphs. If $e(G)=p\binom n2$ and $e(H)=q\binom n2$, then a random relabelling has expected overlap \(pq\binom n2\). They define the positive discrepancy of $G$ with respect to $H$ by
\begin{equation*}
\operatorname{disc}^{+}(G, H) = \max_{G' \cong G} |E(G') \cap E(H)| - pq \binom{n}{2},
\end{equation*}
and the negative discrepancy of $G$ with respect to $H$ by
\begin{equation*}
\operatorname{disc}^{-}(G, H) = pq \binom{n}{2} - \min_{G' \cong G} |E(G') \cap E(H)|,
\end{equation*}
and set $\operatorname{disc}(G, H) = \max\{\operatorname{disc}^{+}(G, H), \operatorname{disc}^{-}(G, H)\}$. They proved that, under the moderate-density assumption ${16}/{n}\le p,q\le 1-{16}/{n}$. Then
\begin{equation*}
\disc^+(G,H)\disc^-(G,H)\ge c\,p^4(1-p)^4q^4(1-q)^4 n^3,
\end{equation*}
and consequently $\disc(G,H)\ge c'\,p^2(1-p)^2q^2(1-q)^2 n^{3/2}$. Thus the natural \(n^{3/2}\) scale persists for intersections of two graphs. Random versions of these intersection problems were studied by Bollob\'as and Scott \cite{BollobasScottRandom2015}, who considered random hypergraphs and tournaments, and by Ma, Naves and Sudakov \cite{MaNavesSudakov2015}, who studied the discrepancy of random graphs and hypergraphs.

Bollob\'as and Scott \cite{BollobasScottHypergraphs2015} extended the definition of relative discrepancy to \(k\)-uniform hypergraphs. If $e(G_k)=p\binom{n}{k}$ and $e(H_k)=q\binom{n}{k}$, then 
\begin{equation*}
\disc(G_k,H_k)=\max_{G_k' \cong G_k}\left|\left|E(G_k')\cap E(H_k)\right|-pq\binom{n}{k}\right|.
\end{equation*}
For \(k=3\), the situation is different from that for graphs. Let \(V=A\cup B\) be a partition, let \(G_3\) consist of all triples meeting both \(A\) and \(B\), and let \(H\) be a Steiner triple system on \(V\). Then every relabelling of \(H_3\) contains exactly \(|A||B|/2\) triples
meeting both parts, and hence $\disc(G_3,H_3)=0$. More recently, Luong-Le, Tran and Yang \cite{LTYrelative2025} showed that zero-discrepancy pairs exist for every \(k\ge3\): for each fixed \(k\ge3\) and all sufficiently large \(n\), there are \(k\)-uniform hypergraphs \(G\) and \(H\), with densities bounded away from \(0\) and \(1\), such that $\disc(G_k,H_k)=0$. 

Bollob\'as and Scott \cite{BollobasScottHypergraphs2015} later developed the corresponding theory for weighted \(k\)-uniform hypergraphs, where zero discrepancy can occur and the corresponding structure can be described using an orthogonal decomposition of the space of weightings.
\begin{theorem}[\textbf{Bollob\'as and Scott} \cite{BollobasScottHypergraphs2015}]\label{thm:2015thm1}
Let \(k \geq 2\). For every \(n \geq 2k\), there are weighted hypergraphs \(w_1, \dots, w_k\) with vertex set \([n]\) such that \(w_i([n]) = 0\) and \(\|w_i\|_1 = \binom{n}{k}\) for every \(i\) and, for \(1 \leq i < j \leq k\), we have
\begin{equation*}
\operatorname{disc}(w_i, w_j) = 0.
\end{equation*}
\end{theorem}
However, they found that things look very different with $k+1$ weighted hypergraphs.
\begin{theorem}[\textbf{Bollob\'as and Scott} \cite{BollobasScottHypergraphs2015}]\label{thm:2015thm2}
For every \(k \geq 1\), there are constants \(c> 0\) such that the following holds. Let \(n \geq 2k\), and suppose that \(w_1, \dots, w_{k+1}\) are weighted \(k\)-uniform hypergraphs on \([n]\) such that \(w_i([n]) = 0\) and \(\|w_i\|_1 = \binom{n}{k}\) for every \(i\). Then there are distinct \(i\) and \(j\) such that
\begin{equation*}
\operatorname{disc}^{+}(w_i, w_j) \operatorname{disc}^{-}(w_i, w_j) \geq c n^{k+1}.
\end{equation*}
\end{theorem}
The directed setting has several additional features. For tournaments, the natural \(n^{3/2}\) scale already appears in the work of Spencer \cite{Spencer1971,Spencer1980} on optimal rankings. Bollob\'as and Scott \cite{BollobasScottHypergraphs2015,BollobasScottRandom2015} asked whether the same scale governs the discrepancy of arbitrary pairs of tournaments and, more generally, what can be said for oriented and directed graphs. Yang and Zeng \cite{YangZengOriented} recently answered the tournament question and proved quantitative lower bounds for oriented graphs of moderate density. Next, we consider the following problem.
\begin{problem}[\textbf{Bollob\'as and Scott} \cite{BollobasScottHypergraphs2015}]
This problem studies how to determine the discrepancy of two directed graphs of the same order.
\end{problem}

Let $D$ and $H$ are two directed graphs of order $n$, with $e(D)=pn(n-1)$ and $e(H)=qn(n-1)$. If we place $D$ and $H$ at random onto the same vertex set, then we expect them to overlap in $pqn(n-1)$ edges. We then study how much or how little we can make them overlap. Let us write $\disc^+(D,H)$ for the largest amount by which we can exceed $pqn(n-1)$, and $\disc^-(D,H)$ for the largest amount less than $pqn(n-1)$. A loopless directed graph is allowed to contain both arcs $(x,y)$ and $(y,x)$. For $\pi\in S_n$ let $D_\pi$ denote the relabelling of $D$ by $\pi$. If $\pi$ is chosen uniformly at random, each arc of $D$ is sent uniformly to one of the $n(n-1)$ ordered pairs of distinct vertices.  We define the positive discrepancy of $D$ with respect to $H$ by
\begin{equation*}\label{eq:discplus-D}
	\disc^+(D,H)=\max_{\pi\in S_n}|E(D_\pi)\cap E(H)|-pqn(n-1),
\end{equation*}
and the negative discrepancy of $G$ with respect to $H$ by
\begin{equation*}\label{eq:discminus-D}
\disc^-(D,H)=pqn(n-1)-\min_{\pi\in S_n}|E(D_\pi)\cap E(H)|,
\end{equation*}
Then
\begin{equation*}
\disc(D,H)=\max\{\disc^+(D,H),\disc^-(D,H)\}.
\end{equation*}
For general directed graphs, however, a uniform two-graph lower bound is impossible. The following elementary example is the basic obstruction.
\begin{remark}
Let $A\subseteq[n]$ with $|A|=r$, and let $D_A$ be the directed graph defined by $E(D_A)=\{(x,y)\in[n]^2:x\in A,y\neq x\}$. Thus every vertex of $A$ has out-degree $n-1$, while every vertex outside $A$ has out-degree 0. If $H$ is d-out-regular, then $\disc(D_A,H)=0$.
\end{remark}
This obstruction motivates a weighted formulation. A weighted directed graph on a vertex set $V$ is a function $w:V^{(2)}\to\R$, where $V^{(2)}$ denotes the ordered pairs of distinct vertices. For weight functions $w,u$ on $[n]$, set
\begin{equation*}
 \ip{w}{u}=\sum_{x\ne y}w_{xy}u_{xy},\qquad\norm{w}_1=\sum_{x\ne y}|w_{xy}|,\qquad w([n])=\sum_{x\ne y}w_{xy}.
\end{equation*}

We work throughout the paper with weighted directed graphs. An \textit{unweighted directed graph} is a subset of \(V^{(2)}\), and can be identified with the weighted directed graph given by the indicator function for its edges. The \textit{density} of \(w\) is defined as $d(w) =w([n])/\left(n(n-1)\right)$. We shall assume that all directed graphs have \(n\) vertices unless otherwise stated (and, in particular, that \(|V| = n\)).  Given two weighted directed graphs \(w, u\) on \(V\), the \textit{intersection} of \(w\) and \(u\) is naturally defined as \(\langle w, u \rangle\), where \(\langle \cdot, \cdot \rangle\) is the standard inner product on \(V^{(2)}\). There is also a natural action of the symmetric group \(S(V)\) on the space of weighted directed graphs, given by \({w_\pi}_{xy} = w_{\pi^{-1}(x)\pi^{-1}(y)}\) (see Section \ref{sec:aux} for notation).

 If $w$ is relabelled uniformly at random, then
\begin{equation}\label{eq:Ewu}
\E_\pi\ip{w_\pi}{u}=d(w)d(u)n(n-1).
\end{equation}
This leads us to define the positive discrepancy of the pair $w,u$ by 
\begin{equation*}
\disc^+(w,u)=\max_\pi\ip{w_\pi}{u}-d(w)d(u)n(n-1),
\end{equation*}
and the negative discrepancy by 
\begin{equation*}
\disc^-(w,u)=d(w)d(u)n(n-1)-\min_\pi\ip{w_\pi}{u},
\end{equation*}
Note that both are nonnegative by (\ref{eq:Ewu}). The discrepancy $\disc(w,u)$ is then defined as
\begin{equation*}
\disc(w,u)=\max\{\disc^+(w,u),\disc^-(w,u)\}.
\end{equation*}
Adding a constant function to either argument does not change discrepancy, so it is natural to work with the centered part $\widetilde w=w-d(w)\cdot\1$. We now state our main results. The first gives the sharp threshold for pairwise zero discrepancy.

\begin{theorem}\label{thm:zero-family}
For every $n\ge4$ there exist four weighted directed graphs $w_1,w_2,w_3,w_4$ with vertex set $[n]$ such that $w_i([n])=0$ and $\norm{w_i}_1=n(n-1)$ for every $i\in[4]$, and such that, for $1\leq i<j\leq 4$, we have
\begin{equation*}
\disc(w_i,w_j)=0.
\end{equation*}
\end{theorem}
\begin{theorem}\label{thm:five}
There exist absolute constants $c,c'>0$ such that the following holds. Let $n\ge4$, and suppose that $w_1,\ldots,w_5$ are weighted directed graphs on $[n]$ such that $w_i([n])=0$ and $\norm{w_i}_1=n(n-1)$ for every $i$. Then there exist distinct $i,j$ such that
\begin{equation*}
\disc^+(w_i,w_j)\disc^-(w_i,w_j)\ge cn^3.
\end{equation*}
In particular, there are $i<j$ such that
\begin{equation*}
\disc(w_i,w_j)\ge c'n^{3/2}.
\end{equation*}
\end{theorem}
Our proof uses two $\mathbb R^2$-valued random vectors $\Phi_1^w$ and $\Phi_2^w$, introduced in Section \ref{sec:decomp}. For $r=1,2$, let $J_r(w,u)=\E\abs{\ip{\Phi_r^w}{\Phi_r^u}}$, where the two random vectors are sampled independently. We prove the following pairwise estimates.
\begin{theorem}\label{thm:profile-disc}
There is an absolute constant $c>0$ such that the following holds. For every $n\ge4$ and every pair of weight functions $w,u:[n]^{(2)}\to\R$, we have
\begin{equation*}
\disc^+(w,u)\disc^-(w,u)\ge c\bigl(n^4J_1(w,u)^2+n^3J_2(w,u)^2\bigr).
\end{equation*}
 In particular,
\begin{equation*}
\disc(w,u)\ge c\bigl(n^2J_1(w,u)+n^{3/2}J_2(w,u)\bigr).
\end{equation*}
\end{theorem}
\begin{theorem}\label{thm:random-overlap}
There is an absolute constant $c>0$ such that, for every $n\ge4$ and every pair of weighted directed graphs $w,u:[n]^{(2)}\to\R$,
\begin{equation*}
\E_\pi\abs{\ip{w_\pi}{u}}\ge c\left(|d(w)d(u)|n(n-1)+n^{3/2}J_1(w,u)+nJ_2(w,u)\right).
\end{equation*}
\end{theorem}
The paper is organized as follows. Section~\ref{sec:aux} introduces notation and auxiliary probabilistic tools. Section~\ref{sec:decomp} develops the orthogonal decomposition, introduces the random vectors $\Phi_1,\Phi_2$ and the parameters $M_r,J_r$. Section~\ref{sec:transpositions} develops the transposition method and establishes the key estimate relating $\gamma(w,u)$ to the interactions between these random vectors. Section~\ref{sec:quantitative} proves the main theorems. Finally, Section~\ref{sec:zero} characterizes zero-discrepancy pairs and proves the sharp four-function result.
\section{Notation and Auxiliary Tools}\label{sec:aux}
We use the following notation. For a finite set \(V\), let $V^{(2)}=\{(x,y)\in V\times V:x\ne y\}$ denote the set of ordered pairs of distinct vertices. When \(V=[n]\), we have \(|V^{(2)}|=n(n-1)\). A weighted directed graph on \(V\) is a function $w:V^{(2)}\to\mathbb R$, write 
\begin{equation*}
w(V)=\sum_{(x,y)\in V^{(2)}}w_{xy},\qquad \|w\|_1=\sum_{(x,y)\in V^{(2)}}|w_{xy}|.
\end{equation*}
For two weight functions \(w,u:V^{(2)}\to\mathbb R\), define
\begin{equation*}
\langle w,u\rangle=\sum_{(x,y)\in V^{(2)}}w_{xy}u_{xy}.
\end{equation*}
The density of \(w\) is
\begin{equation*}
d(w)=\frac{w(V)}{|V|(|V|-1)}.
\end{equation*}
We write \(\mathbf 1\) for the constant weight function on \(V^{(2)}\)
with value \(1\), and define the centered part of \(w\) by $\widetilde w=w-d(w)\cdot\mathbf 1$. Thus \(\widetilde w(V)=0\).

For a weight function \(w:V^{(2)}\to\mathbb R\), its transpose is
defined by $w^\mathsf{T}(x,y)=w(y,x)$. Thus \(w\) is symmetric if \(w=w^\mathsf{T}\), and skew-symmetric if \(w=-w^\mathsf{T}\). Let \(\operatorname{Sym}(V)\) denote the group of all permutations of
\(V\); when \(V=[n]\), we write \(S_n\). For \(\pi\in\operatorname{Sym}(V)\), the relabelling \(w_\pi\) is defined by ${w_\pi}_{xy}=w_{\pi^{-1}(x)\pi^{-1}(y)}$ for every $(x,y)\in V^{(2)}$. Relabelling preserves the total weight and the \(\ell_1\)-norm: $w_\pi(V)=w(V)$ and $\|w_\pi\|_1=\|w\|_1$. It also commutes with transposition, $(w^\mathsf{T})_\pi=(w_\pi)^\mathsf{T}$. Throughout the paper, \((a,b)\), \((a,b,c)\), and \((a,b,c,d)\) denote ordered tuples of pairwise distinct vertices. Thus,
\begin{equation*}
\mathbb E_{(a,b)} f(a,b)=\frac{1}{n(n-1)}\sum_{(a,b)} f(a,b),\qquad \mathbb E_{(a,b,c)} f(a,b,c)
=\frac{1}{n(n-1)(n-2)}\sum_{(a,b,c)} f(a,b,c),
\end{equation*}
It follows easily from \cite{BollobasScottDiscrepancy2006} that we have the following lemma.
\begin{lemma}[\textbf{Bollob\'as and Scott} \cite{BollobasScottDiscrepancy2006}]\label{lem:subset-sum}
	Let $a_1,\ldots,a_m\in\R$. Form a random subset $I\subseteq[m]$ by including each index independently with probability $1/2$. Then
	\begin{equation*}
		\E\abs{\sum_{i\in I}a_i}\ge\frac{1}{4\sqrt{2m}}\sum_{i=1}^m|a_i|.
	\end{equation*}
\end{lemma}
The following simple bound was proved in \cite{BollobasScottHypergraphs2015}.
\begin{proposition}[\textbf{Bollob\'as and Scott} \cite{BollobasScottHypergraphs2015}]\label{prop:translate}
	If $X$ is a random variable with $\E X=0$, and $a\in\R$, then $\E|X+a|\ge \max\left\{\E|X|/2,|a|\right\}$.
\end{proposition}
\section{Weight Decomposition and Structural Parameters}\label{sec:decomp}
In this section, we decompose each centered weight function into four pairwise orthogonal components: symmetric and skew-symmetric parts of first and second order. We then encode these components by two \(\mathbb R^2\)-valued random vectors \(\Phi_1^w\) and \(\Phi_2^w\), and introduce the associated quantities \(M_r\) and \(J_r\). The main purpose is to relate these structural quantities to the \(\ell_1\)-norm of \(w\), which will be used in the transposition
argument of Section~\ref{sec:transpositions}.

For each \(x\in[n]\), define the row and column sums of \(w\) by
\begin{equation*}
 r_x(w)=\sum_{y\ne x}w_{xy},\qquad c_x(w)=\sum_{y\ne x}w_{yx}.
\end{equation*}
Set
\begin{equation*}\label{eq:pq}
 p_x(w)=\frac{r_x(w)+c_x(w)}{2(n-2)},\qquad
 q_x(w)=\frac{r_x(w)-c_x(w)}{2n}.
\end{equation*}
For \(x\ne y\), define \(g_w^+(x,y)=p_x(w)+p_y(w)\) and \(g_w^-(x,y)=q_x(w)-q_y(w)\), and let
\(h_w=w-g_w^+-g_w^-\). Finally, write \(h_w^+=(h_w+h_w^\mathsf{T})/2\) and \(h_w^-=(h_w-h_w^\mathsf{T})/2\).

The following lemma gives the basic orthogonal decomposition of a centered
weight function.
\begin{lemma}\label{lem:canonical-decomposition}
	Let \(n\ge 3\) and $w([n])=0$. Then $w=g_w^++g_w^-+h_w^++h_w^-$, where the four summands have the following properties:
	\begin{enumerate}[label=(\alph*)]
		\item \(g_w^+\) and \(h_w^+\) are symmetric, while \(g_w^-\) and \(h_w^-\) are skew-symmetric. Moreover, \(h_w\), \(h_w^+\), and \(h_w^-\) have zero row and column sums.
		
		\item The four functions $g_w^+$, $g_w^-$, $h_w^+$ and $h_w^-$ are pairwise orthogonal.
		
		\item For every \(\pi\in S_n\), $g_{w_\pi}^{\pm}=(g_w^\pm)_\pi$ and $h_{w_\pi}^{\pm}=(h_w^\pm)_\pi$.
	\end{enumerate}
\end{lemma}
\begin{proof}[\bf Proof]
	By definition, we have
	\begin{equation*}
	g_w^+ + g_w^- + h_w^+ + h_w^-=g_w^+ + g_w^- + h_w=w.
	\end{equation*}
 
 \((a)\) Since \(w([n])=0\), we have
 \begin{equation*}
 	\sum_x r_x(w)=\sum_x c_x(w)=\sum_{(x,y)} w_{xy}=0.
 \end{equation*}
	Hence $\sum_x p_x(w)=\sum_x q_x(w)=0$. For every \(x\in[n]\),
	\begin{equation*}
	\sum_{y\ne x} g_w^+(x,y)=
	\sum_{y\ne x}(p_x(w)+p_y(w))=
	(n-1)p_x(w)+\sum_{y\ne x}p_y(w)=(n-2)p_x(w),
	\end{equation*}
	and similarly
	\begin{equation*}
	\sum_{y\ne x} g_w^-(x,y)=
	\sum_{y\ne x}(q_x(w)-q_y(w))=
	(n-1)q_x(w)-\sum_{y\ne x}q_y(w)=nq_x(w).
	\end{equation*}
Therefore, by the definitions of \(p_x(w)\) and \(q_x(w)\),
\begin{equation}\label{eq:rxw}
	\sum_{y\ne x}\bigl(g_w^+(x,y)+g_w^-(x,y)\bigr)=(n-2)p_x(w)+nq_x(w) =\frac{r_x(w)+c_x(w)}{2}+
\frac{r_x(w)-c_x(w)}{2} =r_x(w).
\end{equation}
Likewise,
\begin{equation*}
\sum_{y\ne x} g_w^+(y,x)=(n-2)p_x(w), \qquad \sum_{y\ne x} g_w^-(y,x)=-nq_x(w),
\end{equation*}
and hence
\begin{equation}\label{eq:cxw}
\sum_{y\ne x}
\bigl(g_w^+(y,x)+g_w^-(y,x)\bigr)=(n-2)p_x(w)-nq_x(w)=
\frac{r_x(w)+c_x(w)}{2}-\frac{r_x(w)-c_x(w)}{2}=c_x(w).
\end{equation}
Since $h_w=w-g_w^+-g_w^-$, (\ref{eq:rxw}) and (\ref{eq:cxw}) imply
\begin{equation*}
\sum_{y\ne x}h_w(x,y)=\sum_{y\ne x}h_w(y,x)=0.
\end{equation*}
Consequently,
\begin{equation*}
\sum_{y\ne x}h_w^+(x,y)=
\frac12\sum_{y\ne x}\bigl(h_w(x,y)+h_w(y,x)\bigr)=0,
\end{equation*}
and
\begin{equation*}
	\sum_{y\ne x}h_w^-(x,y)=\frac12\sum_{y\ne x}\bigl(h_w(x,y)-h_w(y,x)\bigr)=0.
\end{equation*}
By definition, \(g_w^+\) and \(h_w^+\) are symmetric, whereas \(g_w^-\) and \(h_w^-\) are skew-symmetric. The functions $h_w^+$ and $h_w^-$ also have zero column sums.

	\medskip
 $(b)$ Since symmetric functions are orthogonal to skew-symmetric functions, we have $\langle g_w^+,g_w^-\rangle=\langle g_w^+,h_w^-\rangle=\langle h_w^+,g_w^-\rangle=\langle h_w^+,h_w^-\rangle=0$. For the two remaining pairs, part $(a)$ gives
 \begin{equation*}
	\langle g_w^+,h_w^+\rangle=
\sum_{(x,y)}(p_x(w)+p_y(w))h_w^+(x,y)=
\sum_x p_x(w)\sum_{y\ne x}h_w^+(x,y)+
\sum_y p_y(w)\sum_{x\ne y}h_w^+(x,y)=0.
 \end{equation*}
Similarly, we have
\begin{equation*}
	\langle g_w^-,h_w^-\rangle=\sum_{(x,y)}(q_x(w)-q_y(w))h_w^-(x,y)=\sum_x q_x(w)\sum_{y\ne x}h_w^-(x,y)-\sum_y q_y(w)\sum_{x\ne y}h_w^-(x,y)=0.
\end{equation*}
Thus the four summands are pairwise orthogonal.
	
	\medskip
	$(c)$ Let \(\pi\in S_n\). By the definition of relabelling,
	\begin{equation*}
		r_x(w_\pi)=\sum_{y\ne x}w_{\pi^{-1}(x)\pi^{-1}(y)}=r_{\pi^{-1}(x)}(w),
	\end{equation*}
	and
	\begin{equation*}
	c_x(w_\pi)=\sum_{y\ne x}w_{\pi^{-1}(y)\pi^{-1}(x)}=c_{\pi^{-1}(x)}(w).
	\end{equation*}
	It follows that $p_x(w_\pi)=p_{\pi^{-1}(x)}(w)$ and $q_x(w_\pi)=q_{\pi^{-1}(x)}(w)$. Hence, for \(x\ne y\),
	\begin{equation*}
	g_{w_\pi}^+(x, y) = p_x(w_\pi) + p_y(w_\pi)=p_{\pi^{-1}(x)}(w) + p_{\pi^{-1}(y)}(w)=g_w^+(\pi^{-1}(x), \pi^{-1}(y))=(g^+_{w})_{\pi}(x, y)
	\end{equation*} 
	and the same calculation gives \(g_{w_\pi}^-=(g_w^-)_\pi\). Therefore $h_{w_\pi}=w_\pi-g_{w_\pi}^+-g_{w_\pi}^-=(h_w)_\pi$. Since relabelling commutes with transposition, $ h_{w_\pi}^{\pm}=(h_w^\pm)_\pi$. This completes the proof.
\end{proof}
For an ordered pair $a\ne b$, define
\begin{equation*}\label{eq:phi1}
 \Phi_1^w(a,b)=
 \begin{pmatrix}
 \sqrt{\frac{2(n-2)}n}\,\left(p_a(w)-p_b(w)\right)\\[2mm]
 \sqrt2\,\left(q_a(w)-q_b(w)\right)
 \end{pmatrix}\in\R^2.
\end{equation*}
For four distinct vertices $Q=(a,b;c,d)$ define
\begin{equation*}
 C_w(Q)=w_{ac}-w_{bc}-w_{ad}+w_{bd},
\end{equation*}
and let $Q^\mathsf{T}=(c,d;a,b)$. Define
\begin{equation*}\label{eq:phi2}
 \Phi_2^w(Q)=\frac12
 \begin{pmatrix}
 C_w(Q)+C_w(Q^{\mathsf T})\\
 C_w(Q)-C_w(Q^{\mathsf T})
 \end{pmatrix}, \qquad 
 \Phi_2^w(Q^\mathsf{T}) =
 \begin{pmatrix}
 1&0\\
 0&-1
 \end{pmatrix}
 \Phi_2^w(Q).
\end{equation*}
By definitions, we have $C_{g_w^+}(Q)=C_{g_w^-}(Q)=0$. Moreover, since \(h_w^+\) is symmetric and \(h_w^-\) is skew-symmetric, we have $C_{h_w^+}(Q^\mathsf{T})=C_{h_w^+}(Q)$ and $C_{h_w^-}(Q^\mathsf{T})=-C_{h_w^-}(Q)$. Hence
\begin{equation*}
\Phi_2^w(Q)
=\frac{1}{2} \begin{pmatrix} C_{h_w^+}(Q) + C_{h_w^-}(Q) + C_{h_w^+}(Q) - C_{h_w^-}(Q) \\ C_{h_w^+}(Q) + C_{h_w^-}(Q) - C_{h_w^+}(Q) + C_{h_w^-}(Q) \end{pmatrix}=
\begin{pmatrix}
	C_{h_w^+}(Q)\\
	C_{h_w^-}(Q)
\end{pmatrix}.
\end{equation*}
We now define
\begin{equation*}\label{eq:M}
M_r(w)=
\begin{cases}
	|d(w)|,
	& r=0,\\[2mm]
	\displaystyle
	\mathbb E_{(a,b)}
	\bigl\|\Phi_1^w(a,b)\bigr\|_2,
	& r=1,\\[3mm]
	\displaystyle
	\mathbb E_Q
	\bigl\|\Phi_2^w(Q)\bigr\|_2,
	& r=2.
\end{cases}
\end{equation*}
Here \((a,b)\) is chosen uniformly from the ordered pairs of distinct vertices, while \(Q\) is chosen uniformly from the ordered quadruples of distinct vertices. For two weight functions \(w,u\), define
\begin{equation*}\label{eq:J}
J_r(w,u)=
\begin{cases}
	\displaystyle
	\mathbb E_{\substack{(a,b),\; (c,d)}}
	\left|
	\left\langle
	\Phi_1^w(a,b),
	\Phi_1^u(c,d)
	\right\rangle
	\right|,
	& r=1,\\[4mm]
	\displaystyle
	\mathbb E_{\substack{Q,Q'}}
	\left|
	\left\langle
	\Phi_2^w(Q),
	\Phi_2^u(Q')
	\right\rangle
	\right|,
	& r=2.
\end{cases}
\end{equation*}
In each expectation defining \(J_r(w,u)\), the two samples are chosen independently and uniformly.
\begin{lemma}\label{lem:mass}
There are absolute constants $c,C>0$ such that the following holds. For every $n\ge4$, and every weight function $w:[n]^{(2)}\rightarrow \mathbb{R}$,
\begin{equation*}
 c\frac{\norm w_1}{n(n-1)}\le M_0(w)+M_1(w)+M_2(w)\le C\frac{\norm w_1}{n(n-1)}.
\end{equation*}
\end{lemma}
\begin{proof}[\bf Proof]
	Assume first that \(w([n])=0\). Put
	\begin{equation*}
	z_x=\left(\sqrt{\frac{2(n-2)}{n}}\,p_x(w),\sqrt{2}\,q_x(w)\right)\in\mathbb R^2.
	\end{equation*}
	Since $\sum_x p_x(w)=\sum_x q_x(w)=0$, we have $\sum_x z_x=0$. Moreover,
	\(\Phi_1^w(a,b)=z_a-z_b\). Thus, for every fixed \(a\),
	\begin{equation*}
	\frac{1}{n-1}\sum_{b\ne a}(z_a-z_b)=
	z_a-\frac{1}{n-1}\sum_{b\ne a}z_b=
	z_a+\frac{z_a}{n-1}=
	\frac{n}{n-1}z_a .
	\end{equation*}
	Hence, by Jensen's inequality,
	\begin{equation*}
	\frac{n}{n-1}\|z_a\|_2\le\frac{1}{n-1}\sum_{b\ne a}\|z_a-z_b\|_2.
	\end{equation*}
	Summing over \(a\) gives
	\begin{equation*}
		n\sum_a\|z_a\|_2\le
	\sum_{(a,b)}\|z_a-z_b\|_2= \sum_{(a,b)} \|\Phi_1^w(a, b)\|_2=n(n-1) \mathbb{E}_{(a,b)} \|\Phi_1^w(a, b)\|_2=n(n-1) M_1(w),
	\end{equation*}
	and for \(n\ge4\), we have $1\le\sqrt{{2(n-2)}/{n}}<\sqrt2$. Thus $|p_x(w)|+|q_x(w)|\le2\|z_x\|_2$. Therefore
	\begin{equation*}
	\sum_x(|p_x(w)|+|q_x(w)|)\leq2\sum_x\|z_x\|_2\le2(n-1)M_1(w).
	\end{equation*}
	 Since $g_w^+(x,y)+g_w^-(x,y)=p_x(w)+p_y(w)+q_x(w)-q_y(w)$, we obtain
	\begin{align}\label{eq:gw+gw_}
		\|g_w^++g_w^-\|_1&\le\sum_{(x,y)}\bigl(|p_x(w)|+|p_y(w)|+|q_x(w)|+|q_y(w)|
	\bigr)\notag\\&=2(n-1)\sum_x(|p_x(w)|+|q_x(w)|)\notag\\&\le4n(n-1)M_1(w).
	\end{align}
	We next estimate the symmetric second-order part. Fix \(a\ne c\),
	and choose an ordered pair $(b,d)$ uniformly from the ordered pairs of distinct vertices in \([n]\setminus\{a,c\}\). Since \(h_w^+\) is symmetric and has zero row sums,
	\begin{align*}
		&\qquad\mathbb E_{(b,d)}
	C_{h_w^+}(a,b;c,d)=\mathbb E_{(b,d)}\left(h_w^+(a,c)-h_w^+(b,c)-h_w^+(a,d)+h_w^+(b,d)\right)\notag\\&=
	h_w^+(a,c)
	-\frac{1}{n-2}
	\sum_{b\notin\{a,c\}} h_w^+(b,c)
	-\frac{1}{n-2}
	\sum_{d\notin\{a,c\}} h_w^+(a,d)
	+\frac{1}{(n-2)(n-3)}
	\sum_{\substack{b,d\notin\{a,c\}\\ b\ne d}}
	h_w^+(b,d)\notag\\&=
	h_w^+(a,c)
	+\frac{2}{n-2}h_w^+(a,c)
	-\frac{1}{(n-2)(n-3)}
	\left(
	\sum_{b\notin\{a,c\}}h_w^+(b,a)+
	\sum_{b\notin\{a,c\}}h_w^+(b,c)
	\right)\notag\\&=
	h_w^+(a,c)
	+\frac{2}{n-2}h_w^+(a,c)
	+\frac{2}{(n-2)(n-3)}h_w^+(a,c)\notag=\frac{n-1}{n-3}h_w^+(a,c).
	\end{align*}
	Since \((n-1)/(n-3)>1\), Jensen's inequality gives $|h_w^+(a,c)|\le\mathbb E_{(b,d)}
	|C_{h_w^+}(a,b;c,d)|$. Averaging over \(a\ne c\), and using the first coordinate of
	\(\Phi_2^w(Q)\), we have
	 \begin{equation*}
		\mathbb{E}_{(a,c)} |h_w^+(a, c)| = \frac{1}{n(n-1)} \sum_{(a,c)} |h_w^+(a, c)|=\frac{1}{n(n-1)} \|h_w^+\|_1\leq \mathbb{E}_{(a,c)} \mathbb{E}_{(b,d)} |C_{h_w^+}(a, b; c, d)|.
	\end{equation*}
	Then
	\begin{equation}\label{eq:hw+}
	\|h_w^+\|_1\le
	n(n-1)\mathbb E_Q|C_{h_w^+}(Q)|\le
	n(n-1)\mathbb E_Q\|\Phi_2^w(Q)\|_2=n(n-1)M_2(w).
	\end{equation}
	For the skew-symmetric second-order part, define
	\begin{equation*}
	T_{h_w^-}(a,b,c)=h_w^-(a,b)+h_w^-(b,c)+h_w^-(c,a).
	\end{equation*}
Since \(h_w^-\) is skew-symmetric and has zero row sums, we have
	\begin{align*}\label{eq:sumThwabc}
	\sum_{c\notin\{a,b\}}T_{h_w^-}(a,b,c)&=
	\sum_{c\notin\{a,b\}}
	\bigl(
	h_w^-(a,b)+h_w^-(b,c)+h_w^-(c,a)
	\bigr)\notag\\&=
	(n-2)h_w^-(a,b)+
	\sum_{c\notin\{a,b\}}h_w^-(b,c)+
	\sum_{c\notin\{a,b\}}h_w^-(c,a)\notag\\&=
	(n-2)h_w^-(a,b)-h_w^-(b,a)-h_w^-(b,a)\notag\\&=
	(n-2)h_w^-(a,b)+h_w^-(a,b)+h_w^-(a,b)\notag\\&=n h_w^-(a,b).
	\end{align*}
	Therefore,
	\begin{equation*}
	|h_w^-(a,b)|=
	\frac1n\left|
	\sum_{c\notin\{a,b\}}
	T_{h_w^-}(a,b,c)
	\right|\le\frac1n\sum_{c\notin\{a,b\}}|T_{h_w^-}(a,b,c)|.
	\end{equation*}
	Averaging over the ordered pair \(a\ne b\), we obtain
	\begin{equation}\label{eq:EThw-abc}
		\frac{1}{n(n-1)}\|h_w^-\|_1=
	\mathbb E_{(a,b)}|h_w^-(a,b)|\le
	\frac{n-2}{n}
	\mathbb E_{\substack{(a,b,c)}}
	|T_{h_w^-}(a,b,c)|\le
	\mathbb E_{\substack{(a,b,c)}}
	|T_{h_w^-}(a,b,c)|.
	\end{equation}
	For four distinct vertices \(a,b,c,d\), direct expansion gives
	\begin{align*}
2T_{h_w^-}(a,b,c)=-C_{h_w^-}(a,b;c,d)+C_{h_w^-}(a,c;b,d)-C_{h_w^-}(b,c;a,d).
	\end{align*}
	Hence, by the triangle inequality and averaging over \(a,b,c,d\), we have
	\begin{equation}\label{eq:EabcdThw-}
	\mathbb E_{\substack{(a,b,c,d)}}|T_{h_w^-}(a,b,c)|=\mathbb E_{\substack{(a,b,c)}}
	|T_{h_w^-}(a,b,c)|\le
	\frac32\mathbb E_Q|C_{h_w^-}(Q)|\le
	\frac32 M_2(w).
\end{equation}
Combining this with (\ref{eq:EThw-abc}) and (\ref{eq:EabcdThw-}), we obtain
\begin{equation}\label{eq:hw-}
\|h_w^-\|_1\le\frac{3n(n-1)}{2}M_2(w).
\end{equation}
By Lemma~\ref{lem:canonical-decomposition}, we have $w=(g_w^++g_w^-)+h_w^++h_w^-$. Therefore, using (\ref{eq:gw+gw_}), (\ref{eq:hw+}) and (\ref{eq:hw-}), we have
	\begin{align*}
		\|w\|_1&\le
		\|g_w^++g_w^-\|_1
		+\|h_w^+\|_1
		+\|h_w^-\|_1\\&\le
		4n(n-1)M_1(w)
		+n(n-1)M_2(w)
		+\frac32n(n-1)M_2(w)\\&\le
		4 n(n-1)\bigl(M_1(w)+M_2(w)\bigr).
	\end{align*}
	 Then,
	\begin{equation}\label{eq:M1+M2>}
	\frac{\|w\|_1}{4n(n-1)}\le M_1(w)+M_2(w).
	\end{equation}
We now prove the reverse inequality. From the definitions of \(p_x(w)\) and \(q_x(w)\),
\begin{equation*}\label{eq:sumxpxw}
\sum_x|p_x(w)|=\sum_x\left|\frac{r_x(w)+c_x(w)}{2(n-2)}\right|\le\frac{1}{2(n-2)}
\left(\sum_x|r_x(w)|+\sum_x|c_x(w)|\right)\le\frac{\|w\|_1}{n-2},
\end{equation*}
since $\sum_x|r_x(w)|\le\sum_x\sum_{y\ne x}|w(x,y)|=\|w\|_1$, and similarly $\sum_x|c_x(w)|\le\|w\|_1$. Likewise,
\begin{equation*}\label{eq:sumxqxw}
	\sum_x|q_x(w)|=\sum_x\left|
\frac{r_x(w)-c_x(w)}{2n}
\right|\le\frac{1}{2n}\left(\sum_x|r_x(w)|+
\sum_x|c_x(w)|\right)\le\frac{\|w\|_1}{n}.
\end{equation*}
Since \(\|z_x\|_2\le\sqrt2(|p_x(w)|+|q_x(w)|)\), 
\begin{equation}\label{eq:M1<}
M_1(w)=\mathbb E_{(a,b)}\|z_a-z_b\|_2\le\frac2n\sum_x\|z_x\|_2\le
\frac{2\sqrt2}{n}\left(\sum_x|p_x(w)|+\sum_x|q_x(w)|\right)\le \frac{C'\|w\|_1}{n(n-1)}.
\end{equation}
For \(Q=(a,b;c,d)\), the definition of \(\Phi_2^w(Q)\) gives
\begin{align*}\label{eq:1/sqrt2}
	\|\Phi_2^w(Q)\|_2=
\frac12\left(\bigl(C_w(Q)+C_w(Q^\mathsf{T})\bigr)^2+\bigl(C_w(Q)-C_w(Q^\mathsf{T})\bigr)^2\right)^{1/2}\notag\le
\frac1{\sqrt2}\bigl(|C_w(Q)|+|C_w(Q^\mathsf{T})|\bigr).
\end{align*}
Moreover, $|C_w(a,b;c,d)|\le|w(a,c)|+|w(b,c)|+|w(a,d)|+|w(b,d)|$. Under a uniformly random ordered quadruple of distinct vertices, each of the ordered pairs $(a,c)$, $(b,c)$, $(a,d)$ and $(b,d)$ are uniformly distributed over the \(n(n-1)\) ordered pairs of distinct vertices. Hence
\begin{equation*}\label{eq:EQCwQ}
\mathbb E_Q|C_w(Q)|\le\frac{4\|w\|_1}{n(n-1)}.
\end{equation*}
The same estimate holds for \(C_w(Q^T)\), and therefore
\begin{equation*}
M_2(w)=\mathbb E_Q\|\Phi_2^w(Q)\|_2\le\frac1{\sqrt2}
\left(\mathbb E_Q|C_w(Q)|+\mathbb E_Q|C_w(Q^\mathsf{T})|
\right)\le\,\frac{4\sqrt2\|w\|_1}{n(n-1)}.
\end{equation*}
Together with (\ref{eq:M1<}), 
\begin{equation}\label{eq:M1+M2<}
M_1(w)+M_2(w)\le C_0\frac{\|w\|_1}{n(n-1)},
\end{equation}
for centered $w$. Finally, let \(w\) be arbitrary and set $\widetilde w=w-d(w)\mathbf 1$. Then
\begin{equation*}
	n(n-1)|d(w)|=\left|\sum_{x\ne y}w(x,y)\right|\le\sum_{x\ne y}|w(x,y)|=\|w\|_1,
\end{equation*}
and
\begin{equation}\label{eq:w'1<2w1}
	\|\widetilde w\|_1=\sum_{(x,y)}|w(x,y)-d(w)|\le\|w\|_1+n(n-1)|d(w)|\le2\|w\|_1.
\end{equation}
Adding a constant leaves \(M_1\) and \(M_2\) unchanged: indeed,
\(p_x(w+\lambda\mathbf1)\) differs from \(p_x(w)\) by the same
constant for every \(x\), \(q_x(w+\lambda\mathbf1)=q_x(w)\), and
\(C_{w+\lambda\mathbf1}(Q)=C_w(Q)\). Hence $M_1(w)=M_1(\widetilde w)$ and $M_2(w)=M_2(\widetilde w)$. Using (\ref{eq:M1+M2<}) and (\ref{eq:w'1<2w1}),
\begin{equation*}
M_0(w)+M_1(w)+M_2(w)\le\frac{\|w\|_1}{n(n-1)}+C_0\frac{\|\widetilde w\|_1}{n(n-1)}\le
C\frac{\|w\|_1}{n(n-1)}.
\end{equation*}
For the lower bound, since $\|w\|_1\le n(n-1)|d(w)|+\|\widetilde w\|_1$, equation (\ref{eq:M1+M2>}) gives
\begin{equation*}
\frac{\|w\|_1}{n(n-1)}\le M_0(w)+\frac{\|\widetilde w\|_1}{n(n-1)}\le
M_0(w)+4M_1(w)+4M_2(w)\le 4\bigl(M_0(w)+M_1(w)+M_2(w)\bigr).
\end{equation*}
Adjusting the absolute constants completes the proof.
\end{proof}
\section{The Transposition Method}\label{sec:transpositions}
In this section, we study the effect of a transposition on the inner
product \(\langle w_\pi,u_\sigma\rangle\) for every $\pi,\sigma\in S_n$ and derive lower bounds for
\(\gamma(w,u)\) in terms of the structural parameters \(J_1(w,u)\) and
\(J_2(w,u)\).

Fix a transposition $\tau=(xy)$ and define
\begin{equation*}\label{eq:gamma}
 \gamma(w,u)=\E_{\pi,\sigma}\abs{\ip{w_\pi}{u_\sigma}-\ip{w_{\tau\pi}}{u_\sigma}},
\end{equation*}
where $\pi,\sigma$ are chosen independently and uniformly at random. By symmetry, $\gamma(w,u)$ does not depend on the chosen transposition. For distinct $a,b$ and$z\notin\{a,b\}$, define $R_w^{ab}(z)=w_{az}-w_{bz}$, $C_w^{ab}(z)=w_{za}-w_{zb}$ and $A_w^{ab}=w_{ab}-w_{ba}$.
\begin{proposition}\label{prop:single-transposition}
If $\tau=(ab)$, then
\begin{align}
 \ip{w}{u}-\ip{w_{\tau}}{u}
 ={}&\sum_{z\notin\{a,b\}}R_w^{ab}(z)R_u^{ab}(z)+\sum_{z\notin\{a,b\}}C_w^{ab}(z)C_u^{ab}(z)
 +A_w^{ab}A_u^{ab}.\label{eq:single-transposition}
\end{align}
\end{proposition}
\begin{proof}[\bf Proof]
	Since \(\tau=(ab)\) is a transposition, we have \(\tau^{-1}=\tau\). Hence, the relabelled weight function satisfies $(w_\tau)_{xy}=w_{\tau(x)\tau(y)}$. Therefore
	\begin{equation}\label{eq:wu-wu}
		\langle w,u\rangle-\langle w_\tau,u\rangle=\sum_{x\ne y}\bigl(w_{xy}-w_{\tau(x)\tau(y)}\bigr)u_{xy}.
	\end{equation}
For any \(x, y \notin \{a, b\}\), we have \(\tau(x) = x\) and \(\tau(y) = y\), meaning the corresponding terms in (\ref{eq:wu-wu}) are zero. Thus, it suffices to consider only those ordered pairs that contain \(a\) or \(b\). Fix a vertex \(z \notin \{a, b\}\). Under the transposition \((ab)\), the pairs \((a, z)\) and \((b, z)\) are interchanged. Their combined contribution to the sum is
	\begin{equation}\label{eq:az-bz}
	(w_{az}-w_{bz})u_{az}+(w_{bz}-w_{az})u_{bz}=
	(w_{az}-w_{bz})(u_{az}-u_{bz})=R_w^{ab}(z)R_u^{ab}(z).
	\end{equation}
	Similarly, $(z,a)$ and $(z,b)$ are interchanged, contributing
	\begin{equation}\label{eq:za-zb}
	(w_{za}-w_{zb})u_{za}+(w_{zb}-w_{za})u_{zb}=
	(w_{za}-w_{zb})(u_{za}-u_{zb})=C_w^{ab}(z)C_u^{ab}(z).
	\end{equation}
	Finally, the transposition swaps $(a,b)$ and $(b,a)$, yielding a contribution of 
	\begin{equation}\label{eq:ab-ba}
	(w_{ab}-w_{ba})u_{ab}+(w_{ba}-w_{ab})u_{ba}=
	(w_{ab}-w_{ba})(u_{ab}-u_{ba})=A_w^{ab}A_u^{ab}.
	\end{equation}
	Summing the contributions from (\ref{eq:az-bz}) and (\ref{eq:za-zb}) over all \(z\notin\{a,b\}\) and adding (\ref{eq:ab-ba}), we obtain
	\begin{equation*}
	\langle w,u\rangle-\langle w_{\tau},u\rangle=\sum_{z\notin\{a,b\}}R_w^{ab}(z)R_u^{ab}(z)+
	\sum_{z\notin\{a,b\}}C_w^{ab}(z)C_u^{ab}(z)+A_w^{ab}A_u^{ab},
	\end{equation*}
	which completes the proof.
\end{proof}
Consequently, $\gamma(w,u)$ admits the following represectation. Let the fixed transposition be \(\tau=(xy)\). For independent and
uniformly distributed relabellings \(\pi,\sigma\in S_n\), we define $a=\pi^{-1}(x)$, $b=\pi^{-1}(y)$, $c=\sigma^{-1}(x)$ and $d=\sigma^{-1}(y)$. For any vertex \(r\notin\{x,y\}\), we set $z=\pi^{-1}(r)$ and define $\rho(z)=\sigma^{-1}(\pi(z))$. Note that $\rho:[n]\setminus\{a,b\}\longrightarrow[n]\setminus\{c,d\}$ is a bijection. Applying  Proposition~\ref{prop:single-transposition} to the relabelled weights \(w_\pi\) and \(u_\sigma\), we obtain
\begin{equation*}
\langle w_\pi,u_\sigma\rangle-\langle w_{\tau\pi},u_\sigma\rangle=
\sum_{r\notin\{x,y\}}R_{w_\pi}^{xy}(r)R_{u_\sigma}^{xy}(r)+\sum_{r\notin\{x,y\}}
C_{w_\pi}^{xy}(r)C_{u_\sigma}^{xy}(r)+A_{w_\pi}^{xy}A_{u_\sigma}^{xy}.
\end{equation*}
Evaluating the row-difference terms, we have
\begin{equation*}
	R_{w_\pi}^{xy}(r)=(w_\pi)_{xr}-(w_\pi)_{yr}=w_{\pi^{-1}(x)\pi^{-1}(r)}-
w_{\pi^{-1}(y)\pi^{-1}(r)}=w_{az}-w_{bz}=R_w^{ab}(z).
\end{equation*}
Similarly, we have
\begin{equation*}
	R_{u^\sigma}^{xy}(r)=(u_\sigma)_{xr}-(u_\sigma)_{yr}=u_{\sigma^{-1}(x)\sigma^{-1}(r)}
-u_{\sigma^{-1}(y)\sigma^{-1}(r)}=u_{c\rho(z)}-u_{d\rho(z)}=
R_u^{cd}(\rho(z)).
\end{equation*}
Consequently, the product of these terms simplifies to $R_{w_\pi}^{xy}(r)R_{u_\sigma}^{xy}(r)=R_w^{ab}(z)R_u^{cd}(\rho(z))$. An identical argument applies to column-difference terms, giving $C_{w_\pi}^{xy}(r)C_{u_\sigma}^{xy}(r)=C_w^{ab}(z)C_u^{cd}(\rho(z))$. Finally,
\begin{equation*}
	A_{w_\pi}^{xy}=(w_\pi)_{xy}-(w_\pi)_{yx}=w_{\pi^{-1}(x)\pi^{-1}(y)}-
w_{\pi^{-1}(y)\pi^{-1}(x)}=w_{ab}-w_{ba}=A_w^{ab},
\end{equation*}
and similarly $A_{u_\sigma}^{xy}=A_u^{cd}$. Since \(z = \pi^{-1}(r)\), the condition \(r \notin \{x, y\}\) is equivalent to \(z \notin \{a, b\}\). Substituting these terms, we have
\begin{equation}\label{eq:wupi-wu}
\langle w_\pi,u_\sigma\rangle-\langle w_{\tau\pi},u_\sigma\rangle=\sum_{z\notin\{a,b\}}
\Big(R_w^{ab}(z)R_u^{cd}(\rho(z))+C_w^{ab}(z)C_u^{cd}(\rho(z))\Big)+A_w^{ab}A_u^{cd}.
\end{equation}
Because \(\pi\) and \(\sigma\) are chosen independently and uniformly at random, the resulting pairs \((a, b)\) and \((c, d)\) are independent and uniformly distributed ordered pairs of distinct vertices. Conditioned on these two pairs, the composition \(\rho = \sigma^{-1} \circ \pi\) acts as a uniformly random bijection from \([n] \setminus \{a, b\}\) onto \([n] \setminus \{c, d\}\). Therefore, by taking absolute values and expectations on both sides of (\ref{eq:wupi-wu}), we conclude that
\begin{equation}\label{eq:rwu}
	\gamma(w,u)=\mathbb E_{(a,b),\;(c,d),\;\rho}\left|
\sum_{z\notin\{a,b\}}\Big(R_w^{ab}(z)R_u^{cd}(\rho(z))+C_w^{ab}(z)C_u^{cd}(\rho(z))
\Big)+A_w^{ab}A_u^{cd}\right|.
\end{equation}
The following simple bound is proved in \cite{YangZengOriented}.
\begin{lemma}[\textbf{Yang and Zeng} \cite{YangZengOriented}]\label{lem:tauj}
	Let \(w,u:[n]^{(2)}\to\mathbb R\), and let $\{\tau_i:i\in I\}$ be a family of pairwise disjoint transpositions in \(S_n\). For \(J\subseteq I\), write $\tau_J=\prod_{i\in J}\tau_i$, where the order of the product is irrelevant. Choose \(J\subseteq I\)
	randomly by including each \(i\in I\) independently with probability
	\(p\). Then
	\begin{equation*}
\mathbb E_J\left(\langle w,u\rangle-\langle w_{\tau_J},u\rangle\right)=p\sum_{i\in I}\delta_i+p^2\sum_{\substack{i,j\in I\\ i<j}}\beta_{ij},
	\end{equation*}
	where $\delta_i=\langle w,u\rangle-\langle w_{\tau_i},u\rangle$ and $\beta_{ij}=\langle w_{\tau_i},u\rangle+\langle w_{\tau_j},u\rangle-\langle w_{\tau_i\tau_j},u\rangle-\langle w,u\rangle$.
\end{lemma}
Lemma~\ref{lem:tauj} shows that, for weighted directed graphs, the expected change under a random collection of pairwise disjoint transpositions is still a quadratic polynomial in the selection probability. This is the key property used in the proofs of Lemmas~11 and~12 of Bollob\'as and Scott \cite{BollobasScottHypergraphs2015}. Hence their arguments extend to the
present weighted directed setting.
\begin{lemma}[\textbf{Bollob\'as and Scott} \cite{BollobasScottHypergraphs2015}]\label{lem:product-amplification}
There is an absolute constant $c>0$ such that the following holds. For every $n\ge4$ and every pair $w,u$ of functions from $[n]^{(2)}\to \mathbb{R}$, we have
\begin{equation*}
 \disc^+(w,u)\disc^-(w,u)\ge cn^2\gamma(w,u)^2.
\end{equation*}
\end{lemma}
\begin{lemma}[\textbf{Bollob\'as and Scott} \cite{BollobasScottHypergraphs2015}]\label{lem:L1-amplification}
There is an absolute constant $c>0$ such that the following holds. For every $n\ge4$ and every pair $w,u$ of functions from $[n]^{(2)}\to \mathbb{R}$, we have
\begin{equation*}
 \E_\pi\abs{\ip{w_\pi}{u}}\ge c\sqrt n\,\gamma(w,u).
\end{equation*}
\end{lemma}
We next prove a permutation inequality for sums of inner products of two families of vectors.
\begin{lemma}\label{lem:vector-matching}
There exists an absolute constant \(c>0\) such that, for every $m\geq 2$ and every \(x_1,\ldots,x_m,y_1,\ldots,y_m\in\mathbb R^2\), if
\(\rho\) be a uniformly random permutation of \([m]\), then
\begin{equation*}
\mathbb E_\rho\left|\sum_{i=1}^m\langle x_i,y_{\rho(i)}\rangle\right|\ge c\sqrt m\,
\mathbb E_{\substack{(i,j),(k,\ell)}}\left|\langle x_i-x_j,y_k-y_\ell\rangle\right|,
\end{equation*}
where \((i,j)\) and \((k,\ell)\) are chosen independently and
uniformly from the ordered pairs of distinct elements of \([m]\).
\end{lemma}
\begin{proof}[\bf Proof]
Let \(t = \lfloor m/2 \rfloor\). Choose $\alpha$ and $\beta$ independently and uniformly from  \(S_m\) and define
\begin{equation*}
S = \sum_{r=1}^{m} \langle x_{\alpha(r)}, y_{\beta(r)} \rangle.
\end{equation*}
By setting \(i = \alpha(r)\) and \(\rho = \beta \circ \alpha^{-1}\), we can rewrite the sum as \(S = \sum_{i=1}^{m} \langle x_i, y_{\rho(i)} \rangle\). Since \(\alpha\) and \(\beta\) are independent and uniform, \(\rho\) is a uniformly random permutation of \([m]\). Choose a random subset \(J \subseteq [t]\) by including each \(s \in [t]\) independently with probability \(1/2\). Let \(\alpha_J\) be the permutation obtained by transposing \(\alpha(2s - 1)\) and \(\alpha(2s)\) for all \(s \in J\). Define
\begin{equation*}
S_J = \sum_{r=1}^{m} \langle x_{\alpha_J(r)}, y_{\beta(r)} \rangle.
\end{equation*}
For each fixed \(s \in J\), the change in the sum caused by this swap is
\begin{align*}
	\delta_s&=\langle x_{\alpha(2s-1)}, y_{\beta(2s-1)} \rangle + \langle x_{\alpha(2s)}, y_{\beta(2s)} \rangle - \langle x_{\alpha(2s)}, y_{\beta(2s-1)} \rangle - \langle x_{\alpha(2s-1)}, y_{\beta(2s)} \rangle \\
	&= \langle x_{\alpha(2s-1)} - x_{\alpha(2s)}, y_{\beta(2s-1)} - y_{\beta(2s)} \rangle.
\end{align*}
Since the position pairs are disjoint, the total difference is the sum of these changes:
\begin{align*}
S - S_J &=\sum_{r=1}^{m} \left( \langle x_{\alpha(r)}, y_{\beta(r)} \rangle - \langle x_{\alpha_J(r)}, y_{\beta(r)} \rangle \right)\\&=\sum_{s \in J} \left( \langle x_{\alpha(2s-1)}, y_{\beta(2s-1)} \rangle + \langle x_{\alpha(2s)}, y_{\beta(2s)} \rangle 
 - \langle x_{\alpha_J(2s-1)}, y_{\beta(2s-1)} \rangle - \langle x_{\alpha_J(2s)}, y_{\beta(2s)} \rangle \right)\\&= \sum_{s \in J} \langle x_{\alpha(2s-1)} - x_{\alpha(2s)}, y_{\beta(2s-1)} - y_{\beta(2s)} \rangle\\&=\sum_{s \in J} \delta_s.
\end{align*}
Conditioning on \(\alpha\) and \(\beta\), the terms \(\delta_s\) are fixed. Applying Lemma \ref{lem:subset-sum} gives
\begin{equation*}
\mathbb{E}_J |S - S_J| = \mathbb{E}_J \left| \sum_{s \in J} \delta_s \right| \geq \frac{c_0}{\sqrt{t}} \sum_{s=1}^{t} |\delta_s|.
\end{equation*}
Taking expectations over \(\alpha\) and \(\beta\), the pairs \((\alpha(2s - 1), \alpha(2s))\) and \((\beta(2s - 1), \beta(2s))\) are independent and uniformly distributed over all pairs of distinct indices \((i, j)\) and \((k, \ell)\). Thus,
\begin{equation*}
\mathbb{E}_{\alpha, \beta} |\delta_s| = \mathbb{E}_{\substack{(i,j),(k,\ell)}} |\langle x_i - x_j, y_k - y_\ell \rangle|.
\end{equation*}
Therefore,
\begin{equation}\label{eq:ES-SJ}
\mathbb{E}_{\alpha, \beta, J} |S - S_J| \geq \frac{c_0}{\sqrt{t}} \sum_{s=1}^{t} \mathbb{E}_{\alpha, \beta} |\delta_s| = c_0 \sqrt{t} \mathbb{E}_{\substack{(i,j),(k,\ell)}} |\langle x_i - x_j, y_k - y_\ell \rangle|. 
\end{equation}
By the triangle inequality, we have
\begin{equation}\label{eq:2ES}
\mathbb{E}|S - S_J| \leq \mathbb{E}|S| + \mathbb{E}|S_J|.
\end{equation}
Since \(\alpha\) is uniform, \(\alpha_J\) is also uniformly distributed and independent of \(\beta\). Thus, \(S_J\) and \(S\) are identically distributed, implying \(\mathbb{E}|S_J| = \mathbb{E}|S|\). Combining (\ref{eq:ES-SJ}) and (\ref{eq:2ES}), we obtain
\begin{equation*}
2\mathbb{E}|S| \geq c_0 \sqrt{t} \mathbb{E}_{\substack{(i,j),(k,\ell)}} |\langle x_i - x_j, y_k - y_\ell \rangle|.
\end{equation*}
For \(m \geq 2\), \(t = \lfloor m/2 \rfloor \geq m/3\), which implies \(\sqrt{t} \geq \sqrt{m/3}\). Absorbing the absolute constants into \(c\) completes the proof.
\end{proof}
We now combine the preceding estimates to obtain a lower bound for \(\gamma(w,u)\) in terms of \(J_1(w,u)\) and \(J_2(w,u)\).
\begin{lemma}\label{lem:local}
There is an absolute constant $c>0$ such that, for every $n\ge4$,
\begin{equation*}
 \gamma(w,u)\ge c\bigl(nJ_1(w,u)+\sqrt n\,J_2(w,u)\bigr).
\end{equation*}
\end{lemma}
\begin{proof}[\bf Proof]
Adding a constant to the weight functions does not change transposition differences. Moreover, for \(r = 1, 2\), we have \(J_r(w + \lambda \mathbf{1}, u + \mu \mathbf{1}) = J_r(w, u)\), because the differences \(p_a - p_b\), the quantities \(q_a - q_b\), and the \(C_w(Q)\) are unchanged. Hence we may assume \(w([n]) = u([n]) = 0\). For \(z \notin \{a, b\}\), define
\begin{equation*}
 v_w(z)=\bigl(R_w^{ab}(z),C_w^{ab}(z)\bigr)^{\mathsf T},\qquad
\bar v_w=\frac1{n-2}\sum_{z\notin\{a,b\}}v_w(z).
\end{equation*}
Define \(v_u(z)\) and \(\bar{v}_u\) analogously. By (\ref{eq:rwu}), we have 
\begin{align*}
	\gamma(w,u)&=\mathbb E_{(a,b),\;(c,d),\;\rho}\left|
\sum_{z\notin\{a,b\}}\Big(R_w^{ab}(z)R_u^{cd}(\rho(z))+C_w^{ab}(z)C_u^{cd}(\rho(z))
\Big)+A_w^{ab}A_u^{cd}\right|\\&=\mathbb E_{(a,b),\;(c,d),\;\rho}\left|
\sum_{z \notin \{a, b\}} \langle v_w(z), v_u(\rho(z)) \rangle + A_w^{ab} A_u^{cd}\right|\\&=\mathbb E_{(a,b),\;(c,d),\;\rho}\left|(n-2)\langle\bar v_w,\bar v_u\rangle
+\sum_{z\notin\{a,b\}}\left\langle
v_w(z)-\bar v_w,v_u(\rho(z))-\bar v_u\right\rangle+A_w^{ab} A_u^{cd}\right|\\&=\mathbb E_{(a,b),\;(c,d),\;\rho}\left|C_0+Z_\rho\right|,
\end{align*}
where
\begin{equation*}
C_0 = (n - 2) \langle \bar{v}_w, \bar{v}_u \rangle + A_w^{ab} A_u^{cd},\qquad Z_\rho = \sum_{z \notin \{a, b\}} \langle v_w(z) - \bar{v}_w, v_u(\rho(z)) - \bar{v}_u \rangle.
\end{equation*}
For each fixed \(a,b,c,d\), $\mathbb E_\rho Z_\rho=0$. Indeed, for every fixed \(z\),
\begin{equation*}
\mathbb E_\rho\bigl(v_u(\rho(z))-\bar v_u\bigr)=
\frac1{n-2}\sum_{y\notin\{c,d\}}(v_u(y)-\bar v_u)=0.
\end{equation*}
Hence Proposition \ref{prop:translate} implies both \(\mathbb{E}_\rho |C_0 + Z_\rho| \geq |C_0|\) and \(\mathbb{E}_\rho |C_0 + Z_\rho| \geq  \mathbb{E}_\rho |Z_\rho|/2\). Since \(w([n])=0\), we have $r_x(w)=(n-2)p_x(w)+nq_x(w)$, $c_x(w)=(n-2)p_x(w)-nq_x(w)$ and \(A_w^{ab} = 2(q_a(w) - q_b(w)) + 2h_w^-(a, b)\). By definition, we obtain
\begin{equation*}
\sum_{z \notin \{a, b\}} R_w^{ab}(z) = \sum_{z \notin \{a, b\}} (w_{az} - w_{bz}) =(r_a(w) - w_{ab}) - (r_b(w) - w_{ba})= r_a(w) - r_b(w) - A_w^{ab},
\end{equation*}
and
\begin{equation*}
\sum_{z \notin \{a, b\}} C_w^{ab}(z)= \sum_{z \notin \{a, b\}} (w_{za} - w_{zb})= (c_a(w) - w_{ba}) - (c_b(w) - w_{ab}) =  c_a(w) - c_b(w) + A_w^{ab}.
\end{equation*}
Then
\begin{equation*}
\bar{R}_w = (p_a(w) - p_b(w)) + (q_a(w) - q_b(w)) - \frac{2}{n - 2}h_w^-(a, b),
\end{equation*}
and
\begin{equation*}
\bar{C}_w = (p_a(w) - p_b(w)) - (q_a(w) - q_b(w)) + \frac{2}{n - 2}h_w^-(a, b).
\end{equation*}
Substituting these identities gives
\begin{align*}
C_0
&=(n-2)\langle \bar v_w,\bar v_u\rangle
+A_w^{ab}A_u^{cd}\\
&=(n-2)(\bar R_w\bar R_u+\bar C_w\bar C_u)
+A_w^{ab}A_u^{cd}\\
&=2(n-2)\bigl(p_a(w)-p_b(w)\bigr)
\bigl(p_c(u)-p_d(u)\bigr)
+2(n-2)\bigl(q_a(w)-q_b(w)\bigr)
\bigl(q_c(u)-q_d(u)\bigr)\\
&\quad
-4\bigl(q_a(w)-q_b(w)\bigr)h_u^-(c,d)
-4h_w^-(a,b)\bigl(q_c(u)-q_d(u)\bigr)
+\frac{8}{n-2}h_w^-(a,b)h_u^-(c,d)
+A_w^{ab}A_u^{cd}\\
&=2(n-2)\bigl(p_a(w)-p_b(w)\bigr)
\bigl(p_c(u)-p_d(u)\bigr)
+2n\bigl(q_a(w)-q_b(w)\bigr)
\bigl(q_c(u)-q_d(u)\bigr)
\\&\qquad+\frac{4n}{n-2}h_w^-(a,b)h_u^-(c,d)\\
&=n\left\langle
\Phi_1^w(a,b),\Phi_1^u(c,d)
\right\rangle
+\frac{4n}{n-2}h_w^-(a,b)h_u^-(c,d).
\end{align*}
Taking expectations over the uniformly random vertices $a,b,c,d$, we obtain
\begin{equation*}
\gamma(w, u) \geq \mathbb{E}_{(a,b),\;(c,d)}|C_0| \geq nJ_1(w, u) - \frac{4n}{n - 2} \mathbb{E}_{(a,b)}|h_w^-(a, b)| \cdot\mathbb{E}_{(c,d)}|h_u^-(c, d)|.
\end{equation*}
The skew-symmetric part of Lemma~\ref{lem:mass} establishes that \(\mathbb{E}_{(a,b)} |2h_w^-(a, b)| \leq 3\mathbb{E}_Q |C_{h_w^-}(Q)|\).  By the definition of \(\Phi_2\), we have $\Phi_2^w(Q)=\bigl(C_{h_w^+}(Q),C_{h_w^-}(Q)\bigr)^\mathsf{T}$. Since \(Q^\mathsf{T}\) is uniformly distributed whenever \(Q\) is, and $C_{h_w^+}(Q^\mathsf{T})=C_{h_w^+}(Q)$, $C_{h_w^-}(Q^\mathsf{T})=-C_{h_w^-}(Q)$, we have
\begin{align*}
	J_2(w, u) &= \mathbb{E}_{Q, Q'} \left| \langle \Phi_2^w(Q), \Phi_2^u(Q') \rangle \right|= \mathbb{E}_{Q, Q'} \left| C_{h_w^+}(Q) C_{h_u^+}(Q') + C_{h_w^-}(Q) C_{h_u^-}(Q') \right|,
\end{align*}
and
\begin{align*}
	J_2(w, u) = \mathbb{E}_{Q, Q'} \left| \langle \Phi_2^w(Q^\mathsf{T}), \Phi_2^u(Q') \rangle \right| = \mathbb{E}_{Q, Q'} \left| C_{h_w^+}(Q) C_{h_u^+}(Q') - C_{h_w^-}(Q) C_{h_u^-}(Q') \right|.
\end{align*}
Then
\begin{align*}
2J_2(w,u)&= \mathbb{E}_{Q, Q'} \left( \left| C_{h_w^+}(Q) C_{h_u^+}(Q') + C_{h_w^-}(Q) C_{h_u^-}(Q') \right| \right.\left. + \left| C_{h_w^+}(Q) C_{h_u^+}(Q') - C_{h_w^-}(Q) C_{h_u^-}(Q') \right| \right)\\&\ge2\mathbb E_{Q,Q'}\left|C_{h_w^-}(Q)C_{h_u^-}(Q')\right|=2\mathbb E_Q|C_{h_w^-}(Q)|\cdot
\mathbb E_{Q'}|C_{h_u^-}(Q')|.
\end{align*}
By the triangle inequality, we have
\begin{align*}
	\frac{4n}{n-2} \mathbb{E}_{(a,b)} |h_w^-(a, b)| \mathbb{E}_{(c,d)} |h_u^-(c, d)| &\leq \frac{4n}{n-2} \left( \frac{3}{2} \mathbb{E}_Q |C_{h_w^-}(Q)| \right) \cdot\left( \frac{3}{2} \mathbb{E}_{Q'} |C_{h_u^-}(Q')| \right) \\
	&= \frac{9n}{n-2} \mathbb{E}_Q |C_{h_w^-}(Q)|\cdot \mathbb{E}_{Q'} |C_{h_u^-}(Q')| \\
	&\leq \frac{9n}{n-2} J_2(w, u) \\
	&\leq 18 J_2(w, u).
\end{align*}
Then
\begin{equation}\label{eq:r>nj1-18j2}
\gamma(w, u) \geq nJ_1(w, u) - 18J_2(w, u).
\end{equation}
We next obtain a direct lower bound in terms of \(J_2(w,u)\). From Proposition~\ref{prop:translate} and the definition of \(Z_\rho\),
\begin{equation}\label{eq:r=1/2E}
\gamma(w,u)=\mathbb E_{(a,b),\;(c,d)}\mathbb E_\rho |C_0+Z_\rho|\ge\frac12\mathbb E_{(a,b),\;(c,d)}\mathbb E_\rho |Z_\rho|.
\end{equation}
For fixed \(a,b,c,d\), apply Lemma~\ref{lem:vector-matching} to the two centered vector
lists $\{v_w(z)-\bar v_w:z\notin\{a,b\}\}$ and $\{v_u(z)-\bar v_u:z\notin\{c,d\}\}$. Since both lists contain \(n-2\) vectors, we obtain
\begin{equation}\label{eq:cn-2}
\mathbb E_\rho |Z_\rho|\ge c\sqrt{n-2}\,\mathbb E_{\substack{(x,y) ,(r,s)}}
\left|\left\langle v_w(x)-v_w(y),v_u(r)-v_u(s)\right\rangle\right|,
\end{equation}
where \(x\ne y\) are chosen uniformly from \([n]\setminus\{a,b\}\), and \(r\ne s\) are chosen independently and uniformly from \([n]\setminus\{c,d\}\).
By definition, we have
\begin{equation*}
v_w(x)-v_w(y)=
\begin{pmatrix}R_w^{ab}(x)-R_w^{ab}(y)\\C_w^{ab}(x)-C_w^{ab}(y)\end{pmatrix}=
\begin{pmatrix}C_w(a,b;x,y)\\C_w(x,y;a,b)\end{pmatrix}.
\end{equation*}
Similarly, we obtain
\begin{equation*}
v_u(r)-v_u(s)=
\begin{pmatrix}C_u(c,d;r,s)\\C_u(r,s;c,d)\end{pmatrix}.
\end{equation*}
Hence
\begin{equation}\label{eq:vvuu=c}
\left\langle v_w(x)-v_w(y),v_u(r)-v_u(s)\right\rangle=
C_w(a,b;x,y)C_u(c,d;r,s)+C_w(x,y;a,b)C_u(r,s;c,d).
\end{equation}
On the other hand, by the definition of \(\Phi_2\), we have
\begin{equation}\label{eq:2phi=c}
	2\left\langle\Phi_2^w(a,b;x,y),\Phi_2^u(c,d;r,s)\right\rangle=
C_w(a,b;x,y)C_u(c,d;r,s)+C_w(x,y;a,b)C_u(r,s;c,d).
\end{equation}
Combining (\ref{eq:vvuu=c}) and (\ref{eq:2phi=c}), we have
\begin{equation}\label{eq:v-vv-v}
\left|\left\langle v_w(x)-v_w(y),v_u(r)-v_u(s)\right\rangle\right|
=2\left|\left\langle\Phi_2^w(a,b;x,y),\Phi_2^u(c,d;r,s)\right\rangle\right|.
\end{equation}
Substituting (\ref{eq:v-vv-v}) into (\ref{eq:cn-2}) and then into (\ref{eq:r=1/2E}) gives
\begin{equation*}
\gamma(w,u)\ge c\sqrt{n-2}\,\mathbb E_{\substack{(a,b) \\ (c,d)}}\mathbb E_{\substack{(x,y) \\ (r,s)}}\left|
\left\langle\Phi_2^w(a,b;x,y),\Phi_2^u(c,d;r,s)\right\rangle\right|.
\end{equation*}
When \(a,b\) are chosen uniformly with \(a\ne b\), and then
\(x,y\) are chosen uniformly and distinctly from
\([n]\setminus\{a,b\}\), the ordered quadruple
\((a,b;x,y)\) is uniform over all ordered quadruples of distinct
vertices. The same holds independently for \((c,d;r,s)\).
Therefore
\begin{equation*}
\mathbb E_{\substack{(a,b;x,y) \\ (c,d;r,s)}}\left|\left\langle\Phi_2^w(a,b;x,y),\Phi_2^u(c,d;r,s)\right\rangle\right|
=J_2(w,u).
\end{equation*}
Thus
\begin{equation}\label{eq:r>cnj2}
\gamma(w,u)\ge c\sqrt{n-2}\,J_2(w,u)\ge c\sqrt n\,J_2(w,u),
\end{equation}
after adjusting the absolute constant. It remains to combine (\ref{eq:r>nj1-18j2}) and (\ref{eq:r>cnj2}). Suppose first that $nJ_1(w,u)\ge36J_2(w,u)$. Then (\ref{eq:r>nj1-18j2}) gives $\gamma(w,u)\ge nJ_1(w,u)-18J_2(w,u)\ge nJ_1(w,u)/2$. Together with (\ref{eq:r>cnj2}), this implies $\gamma(w,u)\ge c\bigl(nJ_1(w,u)+\sqrt n\,J_2(w,u)\bigr)$. 

Now suppose that $nJ_1(w,u)<36J_2(w,u)$. Since \(n\ge4\), $36J_2(w,u)\le18\sqrt n\,J_2(w,u)$, and hence $nJ_1(w,u)+\sqrt n\,J_2(w,u)\le19\sqrt n\,J_2(w,u)$. The desired bound now follows directly from (\ref{eq:r>cnj2}), after adjusting
the absolute constant. This completes the proof.
\end{proof}

\section{Proofs of Main Theorems}\label{sec:quantitative}
In this section, we combine the transposition estimates from Section~\ref{sec:transpositions} with a auxiliary lemma to prove Theorems~\ref{thm:five}--\ref{thm:random-overlap}.
\begin{proof}[\bf Proof of Theorem~\ref{thm:profile-disc}]
	By Lemmas~\ref{lem:product-amplification} and~\ref{lem:local}, we have
	\begin{equation*}
		\disc^+(w,u)\disc^-(w,u)\ge c n^2\gamma(w,u)^2\ge c\bigl(n^2J_1(w,u)+n^{3/2}J_2(w,u)\bigr)^2.
	\end{equation*}
	Moreover,
	\begin{equation*}
	\disc(w,u)\ge c\bigl(n^2J_1(w,u)+n^{3/2}J_2(w,u)\bigr),
	\end{equation*}
	after adjusting the absolute constant. This completes the proof.
\end{proof}
\begin{proof}[\bf Proof of Theorem \ref{thm:random-overlap}]
Assume first that \(w([n])=u([n])=0\). By Lemmas~\ref{lem:L1-amplification} and~\ref{lem:local},
\begin{equation}\label{eq:Ecnj1nj2}
\mathbb E_\pi\bigl|\langle w_\pi,u\rangle\bigr|\ge c\sqrt n\,\gamma(w,u)\geq c\bigl(
n^{3/2}J_1(w,u)+nJ_2(w,u)\bigr).
\end{equation}
Now let \(w,u\) be arbitrary, and write $\widetilde w=w-d(w)\cdot\mathbf1$ and $\widetilde u=u-d(u)\cdot\mathbf1$. Then \(\widetilde{w}([n])=\widetilde{u}([n])=0\) and
\begin{equation*}
\langle w_\pi,u\rangle=\left\langle d(w)\cdot\mathbf1+\widetilde w_\pi,d(u)\cdot\mathbf1+\widetilde u
\right\rangle=n(n-1)d(w)d(u)+\langle\widetilde w_\pi,\widetilde u\rangle.
\end{equation*}
Since \(\widetilde w\) and \(\widetilde u\) are centered, we have $\mathbb E_\pi\langle\widetilde w^\pi,\widetilde u\rangle=0$. Moreover, adding constants does not change \(\Phi_1\) or \(\Phi_2\), and hence for $r=1,2$, we have $J_r(\widetilde w,\widetilde u)=J_r(w,u)$. Therefore Proposition \ref{prop:translate} and (\ref{eq:Ecnj1nj2}) gives
\begin{align*}
\mathbb E_\pi|\langle w_\pi,u\rangle|
&\ge\max\left\{n(n-1)|d(w)d(u)|,\,\frac12\mathbb E_\pi|\langle\widetilde w_\pi,\widetilde u\rangle|\right\}\\&\geq c\Big(n(n-1)|d(w)d(u)|+n^{3/2}J_1(w,u)+nJ_2(w,u)\Big).
\end{align*}
This completes the proof.
\end{proof}
To prove the Theorem~\ref{thm:five}, we use the following lemma, which guarantees that among three \(\mathbb R^2\)-valued random vectors with sufficiently large expected norms, two have a large expected inner product.
\begin{lemma}\label{lem:packing}
Let $a>0$, $X_1,X_2,X_3$ be $\mathbb R^2$-valued random vectors such that for $i=1,2,3$, $\E\norm{X_i}_2\ge a$. Then there exist $i<j$ such that, for independent samples of $X_i$ and $X_j$,
\begin{equation*}
 \E\abs{\ip{X_i}{X_j}}\ge\frac{a^2}{4}.
\end{equation*}
\end{lemma}
\begin{proof}[\bf Proof]
	The assertion is trivial if \(a=0\), so assume \(a>0\). Set \(m_i=\mathbb E_{X_i}\|X_i\|_2\ge a\), and define
	\begin{equation*}
	B_i=\frac1{m_i}\mathbb E_{X_i}\left(\frac{X_iX_i^\mathsf{T}}{\|X_i\|_2}\mathbf 1_{\{X_i\ne0\}}\right).
	\end{equation*}
	For every \(z\in\mathbb R^2\),
	\begin{equation*}
	z^\mathsf{T}B_iz=\frac1{m_i}\mathbb E_{X_i}\left(\frac{(z^\mathsf{T}X_i)^2}{\|X_i\|_2}\mathbf 1_{\{X_i\ne0\}}\right)\ge0,
	\end{equation*}
	so \(B_i\succeq0\). Moreover,
	\begin{equation}\label{eq:trB=1}
	\operatorname{tr}B_i=\frac1{m_i}\mathbb E_{X_i}\left(\frac{\operatorname{tr}(X_iX_i^\mathsf{T})}{\|X_i\|_2}
	\mathbf 1_{\{X_i\ne0\}}\right)=\frac1{m_i}\mathbb E_{X_i}\|X_i\|_2=1.
	\end{equation}
For independent samples \(X_i\) and \(X_j\), independently gives
\begin{align*}
m_im_j\operatorname{tr}(B_iB_j)&=\mathbb E_{X_i,X_j}\left(\frac{\operatorname{tr}(X_iX_i^\mathsf{T} X_jX_j^\mathsf{T})}{
\|X_i\|_2\|X_j\|_2}\mathbf 1_{\{X_i\ne0,X_j\ne0\}}\right)\\&=\mathbb E_{X_i,X_j}\left(\frac{\langle X_i,X_j\rangle^2}{\|X_i\|_2\|X_j\|_2}\mathbf 1_{\{X_i\ne0,X_j\ne0\}}\right)\\&\le
\mathbb E_{X_i,X_j}|\langle X_i,X_j\rangle|,
\end{align*}
Hence
\begin{equation}\label{eq:E>a^2trBB}
\mathbb E_{X_i,X_j}|\langle X_i,X_j\rangle|\ge a^2\operatorname{tr}(B_iB_j).
\end{equation}
It remains to show that $\operatorname{tr}(B_iB_j)\ge1/4$ for some \(i<j\). Since \(B_1+B_2+B_3\succeq0\) and (\ref{eq:trB=1}), we have
\begin{equation*}
\|B_1+B_2+B_3\|_\mathrm{F}^2\ge\frac{(\operatorname{tr}(B_1+B_2+B_3))^2}{2}=\frac92.
\end{equation*}
On the other hand, since \(B_i\succeq0\) and (\ref{eq:trB=1}), $\|B_i\|_\mathrm{F}^2\le1$. Therefore
\begin{equation*}
\frac92\le\|B_1+B_2+B_3\|_\mathrm{F}^2=\sum_{i=1}^3\|B_i\|_\mathrm{F}^2+2\sum_{i<j}\operatorname{tr}(B_iB_j)\le
3+2\sum_{i<j}\operatorname{tr}(B_iB_j).
\end{equation*}
Thus
\begin{equation*}
\sum_{i<j}\operatorname{tr}(B_iB_j)\ge\frac34.
\end{equation*}
There are three pairs, so for some \(i<j\), $\operatorname{tr}(B_iB_j)\ge1/4$. Substituting this into (\ref{eq:E>a^2trBB}) gives $\mathbb E|\langle X_i,X_j\rangle|\ge{a^2}/{4}$.
\end{proof}
\begin{proof}[\bf Proof of Theorem \ref{thm:five}]
By Lemma~\ref{lem:mass}, there exists an absolute constant \(c_0>0\) such that every \(w_s([n])=0\) with \(\|w_s\|_1=n(n-1)\) satisfies $M_1(w_s)+M_2(w_s)\ge c_0$. Hence, for each \(s\), at least one of \(M_1(w_s)\) and \(M_2(w_s)\) is at least \(c_0/2\). By the pigeonhole principle, there exist three functions, say \(w_{s_1},w_{s_2},w_{s_3}\), and some \(r\in\{1,2\}\) such that for $k=1,2,3$, $M_r(w_{s_k})\ge{c_0}/{2}$.

Apply Lemma~\ref{lem:packing} to the three random vectors $\Phi_r^{w_{s_1}}$, $\Phi_r^{w_{s_2}}$ and $\Phi_r^{w_{s_3}}$. There exist two of them, say \(w_i,w_j\), such that
\begin{equation*}
J_r(w_i,w_j)=\mathbb E\left|\left\langle\Phi_r^{w_i},\Phi_r^{w_j}\right\rangle
\right|\ge\frac14\left(\frac{c_0}{2}\right)^2=\frac{c_0^2}{16}.
\end{equation*}
If \(r=1\), Theorem~\ref{thm:profile-disc} gives
\begin{equation*}
\disc^+(w_i,w_j)\disc^-(w_i,w_j)\ge c n^4J_1(w_i,w_j)^2\ge c n^4.
\end{equation*}
If \(r=2\), then
\begin{equation*}
	\disc^+(w_i,w_j)\disc^-(w_i,w_j)\ge c n^3J_2(w_i,w_j)^2\ge c n^3.
\end{equation*}
After adjusting the absolute constant, both cases give
\begin{equation*}
\disc^+(w_i,w_j)\disc^-(w_i,w_j)\ge c n^3.
\end{equation*}
For an absolute constant \(c'>0\), we have $\disc(w_i,w_j)\ge c'n^{3/2}$. This completes the proof.
\end{proof}

\section{Orthogonal Sets of Weightings}\label{sec:zero}
In this section, we study the structure of pairwise zero-discrepant weight functions. Using the orthogonal decomposition from Section~\ref{sec:decomp}, we characterize when two centered weight functions have zero discrepancy. As consequences, we prove that a pairwise zero-discrepant family has size at most four and describe the structure of the extremal four-element families.

Define the five mutually orthogonal subspaces $ V_0=\{\lambda\1:\lambda\in\R\}$,
\begin{equation*}
 V_1^+=\left\{w:w(x,y)=p_x+p_y,\ \sum_xp_x=0\right\},\qquad V_1^-=\left\{w:w(x,y)=q_x-q_y,\ \sum_xq_x=0\right\},
\end{equation*}
\begin{equation*}
  V_2^+=\left\{w:w=w^\mathsf{T},\ \sum_{y\ne x}w(x,y)=0\ \text{for every }x\right\}, 
\end{equation*}
and
\begin{equation*}
V_2^-=\left\{w:w=-w^\mathsf{T},\ \sum_{y\ne x}w(x,y)=0\ \text{for every }x\right\}.
\end{equation*}
By Lemma \ref{lem:canonical-decomposition}, we have
\begin{equation*}\label{eq:five-spaces}
\mathbb R^{[n]^{(2)}}=V_0\oplus V_1^+\oplus V_1^-\oplus V_2^+\oplus V_2^-.
\end{equation*}
Each subspace is preserved by relabelling. Hence for centered \(w,u\) and every \(\pi\in S_n\),
\begin{equation}\label{eq:channel-decomp}
\ip{w_\pi}{u}=\ip{(g_w^++g_w^-)_\pi}{g_u^++g_u^-}+\ip{(h_w^+)_\pi}{h_u^+}+\ip{(h_w^-)_\pi}{h_u^-}.
\end{equation}
For centered \(w\), define $\alpha_w=\sqrt{2(n-2)}\bigl(p_1(w),\ldots,p_n(w)\bigr)$ and $\beta_w=\sqrt{2n}\bigl(q_1(w),\ldots,q_n(w)\bigr)$. Since
\begin{equation*}
\sum_xp_x(w)=\sum_xq_x(w)=0,
\end{equation*}
both vectors have coordinate sum zero. Define
\begin{equation*}
\Gamma_w=\begin{pmatrix}\|\alpha_w\|_2^2&\langle\alpha_w,\beta_w\rangle\\\langle\alpha_w,\beta_w\rangle&\|\beta_w\|_2^2\end{pmatrix}.
\end{equation*}
Then \(\Gamma_w\) is positive semidefinite, and $\operatorname{rank}(\Gamma_w)=\dim\operatorname{span}\{\alpha_w,\beta_w\}$.
\begin{lemma}\label{lem:second-moment}
	Let \(a,b,c,d\in\mathbb R^n\) satisfy $\sum_i a_i=\sum_i b_i=\sum_i c_i=\sum_i d_i=0$, and let \(\pi\) be uniform in \(S_n\). Then
	\begin{equation*}
	\mathbb E_\pi\bigl(\langle a_\pi,b\rangle\langle c_\pi,d\rangle\bigr)=\frac{\langle a,c\rangle\langle b,d\rangle}{n-1}.
	\end{equation*}
\end{lemma}
\begin{proof}[\bf Proof]
	Expanding the two inner products gives
	\begin{equation*}
	\mathbb E_\pi\bigl(\langle a_\pi,b\rangle\langle c_\pi,d\rangle\bigr)= \mathbb{E}_\pi \left( \left( \sum_i a_{\pi^{-1}(i)} b_i \right) \left( \sum_j c_{\pi^{-1}(j)} d_j \right) \right)=
	\sum_{i,j}b_i d_j\,\mathbb E_\pi\left(a_{\pi^{-1}(i)}c_{\pi^{-1}(j)}\right).
	\end{equation*}
	For \(i=j\),
	\begin{equation*}
	\mathbb E_\pi\left(a_{\pi^{-1}(i)}c_{\pi^{-1}(i)}\right)= \sum_{r=1}^n a_r c_r \mathbb{P}\left(\pi^{-1}(i) = r\right)=\frac{1}{n} \sum_{r=1}^n a_r c_r=\frac{\langle a,c\rangle}{n}.
	\end{equation*}
	For \(i\ne j\),
	\begin{equation*}
\mathbb E_\pi\bigl(a_{\pi^{-1}(i)}c_{\pi^{-1}(j)}\bigr)=\frac{1}{n(n-1)}\sum_{r \neq s} a_r c_s =\frac{1}{n(n-1)}\left(\sum_{r,s} a_r c_s - \sum_{r} a_r c_r\right)=-\frac{\langle a,c\rangle}{n(n-1)}.
	\end{equation*}
	Since
	\begin{equation*}
	\sum_{i \neq j} b_i d_j = \sum_{i,j} b_i d_j - \sum_{i} b_i d_i = \left( \sum_{i} b_i \right) \left( \sum_{j} d_j \right) - \sum_{i} b_i d_i = - \langle b, d \rangle.
	\end{equation*}
	we obtain
	\begin{equation*}
	\mathbb E_\pi\bigl(\langle a_\pi,b\rangle\langle c_\pi,d\rangle\bigr)= \frac{\langle a, c \rangle}{n} \sum_{i} b_i d_i - \frac{\langle a, c \rangle}{n(n-1)} \sum_{i \neq j} b_i d_j=\frac{\langle a,c\rangle\langle b,d\rangle}{n}+\frac{\langle a,c\rangle\langle b,d\rangle}{n(n-1)}=\frac{\langle a,c\rangle\langle b,d\rangle}{n-1}.
	\end{equation*}
This proves the lemma.
\end{proof}
\begin{theorem}\label{thm:zero-classification}
Let $n\ge4$, and let $w([n])=u([n])=0$. Then $\disc(w,u)=0$ if and only if all three conditions hold:
\begin{enumerate}[label=(\roman*)]
\item $h_w^+=0$ or $h_u^+=0$;\label{itemi}
\item $h_w^-=0$ or $h_u^-=0$;\label{itemii}
\item $\operatorname{tr}(\Gamma_w\Gamma_u)=0$, equivalently \(\operatorname{col}(\Gamma_w)\perp
\operatorname{col}(\Gamma_u)\) in \(\mathbb R^2\).\label{itemiii}
\end{enumerate}
Consequently, if $\mathcal F$ is a pairwise zero-discrepant family of nonzero centered weight functions, then
\begin{equation*}\label{eq:resource}
 \sum_{w\in\mathcal F}
 \left(\rank(\Gamma_w)+\mathbf1_{\{h_w^+\ne0\}}+\mathbf1_{\{h_w^-\ne0\}}\right)\le4,
\end{equation*}
and hence $|\mathcal F|\le4$.
\end{theorem}
\begin{proof}[\bf Proof]
Because $w,u$ are centered, for every $\pi\in S_n$
\begin{equation}\label{eq:disc-zero-all}
 \disc(w,u)=0\quad\Longleftrightarrow\quad \ip{w_\pi}{u}=0.
\end{equation}
\begin{claim}\label{claim:i}
	If $h_w^+\neq 0$, then $h_u^+=0$.
\end{claim}
\begin{proof}
For four distinct vertices \(Q=(a,b;c,d)\), let \(R_Q\in V_2^+\) 
\begin{equation*}
R_Q(x, y) = \begin{cases} 
	1, & (x, y) \in \{(a, c), (c, a), (b, d), (d, b)\}, \\
	-1, & (x, y) \in \{(a, d), (d, a), (b, c), (c, b)\}, \\
	0, & \text{otherwise}.
\end{cases}
\end{equation*}
 We first show that these vectors span \(V_2^+\). Suppose \(z\in V_2^+\) is orthogonal to every \(R_Q\). Since \(z\) is symmetric,
\begin{equation*}
\frac12\langle z,R_Q\rangle=\frac{1}{2} \sum_{x \neq y} z(x, y) R_Q(x, y)=z(a,c)-z(b,c)-z(a,d)+z(b,d)=0.
\end{equation*}
Thus, for fixed \(a\ne b\), $z(a,c)-z(b,c)=z(a,d)-z(b,d)$ whenever \(c,d\notin\{a,b\}\). Summing over \(d\notin\{a,b,c\}\) gives
\begin{align*}
	(n-3)\left(z(a,c)-z(b,c)\right)&=\sum_{d\notin\{a,b,c\}}\bigl(z(a,d)-z(b,d)\bigr)\\&=-z(a,b)-z(a,c)+z(b,a)+z(b,c)\\&=-\left(z(a,c)-z(b,c)\right).
\end{align*}
Then $(n-2)\left(z(a,c)-z(b,c)\right)=0$. Hence $z(a,c)=z(b,c)$ whenever \(a,b,c\) are distinct. Fix \(x\in[n]\) and let \(y,z\ne x\). By symmetry and the preceding identity, $z(x,y)=z(y,x)=z(z,x)=z(x,z)$. Thus all off-diagonal entries in each row are equal. Hence, for any fixed \(y\ne x\),
\begin{equation*}
\sum_{z\ne x} z(x,z)=(n-1)z(x,y)=0.
\end{equation*}
Therefore \(z(x,y)=0\) for every \(x\ne y\), and hence \(z=0\).
Consequently, span$\{R_Q^\pi:\pi\in S_n\}=V_2^+$.

Suppose that \(h_w^+\ne0\). Since span$\{R_Q^\pi:\pi\in S_n\}=V_2^+$ and $\langle h_w^+,R_Q\rangle=2C_{h_w^+}(Q)$, there exists an ordered quadruple \(Q=(a,b;c,d)\) such that
	\(C_{h_w^+}(Q)\ne0\). Let \(\tau_1=(ab)\) and \(\tau_2=(cd)\), and define
\begin{equation*}
	\Delta_Qw=w-w_{\tau_1}-w_{\tau_2}+w_{\tau_1\tau_2}.
\end{equation*}
Then for $x\neq y$
\begin{align*}
	(\Delta_Qg_w^+)(x,y)&=g_w^+(x,y)-(g_w^+)_{\tau_1}(x,y)-(g_w^+)_{\tau_2}(x,y)+(g_w^+)_{\tau_1\tau_2}(x,y)\\&=(p_x+p_y)-(p_{\tau_1(x)}+p_{\tau_1(y)})-(p_{\tau_2(x)}+p_{\tau_2(y)})+(p_{\tau_1\tau_2(x)}+p_{\tau_1\tau_2(y)})\\&=\bigl(p_x-p_{\tau_1(x)}-p_{\tau_2(x)}+p_{\tau_1\tau_2(x)}\bigr)+\bigl(p_y-p_{\tau_1(y)}-p_{\tau_2(y)}+p_{\tau_1\tau_2(y)}
	\bigr)\\&=0.
\end{align*}
Similarly, $(\Delta_Qg_w^-)(x,y)=0$. Using $w=g_w^++g_w^-+h_w^++h_w^-$, we have
\begin{equation*}
\Delta_Q w=w-w_{\tau_1}-w_{\tau_2}+w_{\tau_1\tau_2}=\Delta_Q g_w^++\Delta_Q g_w^-+\Delta_Q h_w^++\Delta_Q h_w^-=\Delta_Q h_w^++\Delta_Q h_w^-,
\end{equation*}
For the symmetric part,
\begin{align*}
	(\Delta_Q h_w^+)(a, c) &= h_w^+(a, c) - (h_w^+)_{\tau_1}(a, c) - (h_w^+)_{\tau_2}(a, c) + (h_w^+)_{\tau_1 \tau_2}(a, c) \\
	&= h_w^+(a, c) - h_w^+(\tau_1(a), \tau_1(c)) - h_w^+(\tau_2(a), \tau_2(c))+ h_w^+(\tau_1\tau_2(a), \tau_1 \tau_2(c)) \\
	&= h_w^+(a, c) - h_w^+(b, c) - h_w^+(a, d) + h_w^+(b, d) \\
	&= C_{h_w^+}(a, b; c, d) \\
	&= C_{h_w^+}(Q).
\end{align*}
	The same calculation at the remaining affected coordinates gives
	\begin{align*}
		(\Delta_Q h_w^+)(a, d) &= h_w^+(a, d) - h_w^+(b, d) - h_w^+(a, c) + h_w^+(b, c) = -C_{h_w^+}(Q), \\
		(\Delta_Q h_w^+)(b, c) &= h_w^+(b, c) - h_w^+(a, c) - h_w^+(b, d) + h_w^+(a, d)= -C_{h_w^+}(Q), \\
		(\Delta_Q h_w^+)(b, d) &= h_w^+(b, d) - h_w^+(a, d) - h_w^+(b, c) + h_w^+(a, c) = C_{h_w^+}(Q).
	\end{align*}
	Then $\Delta_Qh_w^+=C_{h_w^+}(Q)R_Q$. For the skew-symmetric part, direct calculation gives $(\Delta_Qh_w^-)(a,c)=(\Delta_Qh_w^-)(b,d)=C_{h_w^-}(Q)$ and $(\Delta_Qh_w^-)(a,d)=(\Delta_Qh_w^-)(b,c)=-C_{h_w^-}(Q)$. Let \(\sigma=(ac)(bd)\) 
	\begin{align*}
		\left((\Delta_Q h_w^-)_\sigma\right)(a, c) &= (\Delta_Q h_w^-)(c, a) \\
		&= h_w^-(c, a) - h_w^-(c, b) - h_w^-(d, a) + h_w^-(d, b) \\
		&= -h_w^-(a, c) + h_w^-(b, c) + h_w^-(a, d) - h_w^-(b, d) \\
		&= -\left( h_w^-(a, c) - h_w^-(b, c) - h_w^-(a, d) + h_w^-(b, d) \right) \\
		&= -(\Delta_Q h_w^-)(a, c).
	\end{align*}
	Similarly, we have $\bigl((\Delta_Q h_w^-)_\sigma\bigr)(a,d)=-(\Delta_Q h_w^-)(a,d)$, $\bigl((\Delta_Q h_w^-)_\sigma\bigr)(b,c)=-(\Delta_Q h_w^-)(b,c)$ and $\bigl((\Delta_Q h_w^-)_\sigma\bigr)(b,d)=-(\Delta_Q h_w^-)(b,d)$. Then $(\Delta_Qh_w^-)_\sigma=-\Delta_Qh_w^-$. Similarly, we have $(\Delta_Q h_w^+)_\sigma = \Delta_Q h_w^+$. Then
	\begin{equation*}
		\Delta_Qw+(\Delta_Qw)_\sigma=\Delta_Qh_w^++(\Delta_Qh_w^+)_\sigma+\Delta_Qh_w^-
	+(\Delta_Qh_w^-)_\sigma=2C_{h_w^+}(Q)R_Q.
	\end{equation*}
 By (\ref{eq:disc-zero-all}), every relabelling of $\Delta_Qw+(\Delta_Qw)_{\sigma}$ is orthogonal to $u$, since this vector is a finite linear combination of relabellings of $w$. Hence, for every $\pi$,
 \begin{equation*}
\left\langle\bigl(\Delta_Qw+(\Delta_Qw)_\sigma\bigr)_\pi,u\right\rangle=2C_{h_w^+}(Q)\langle R_Q^\pi,u\rangle=0.
 \end{equation*}
Since \(C_{h_w^+}(Q)\ne0\), we have $\langle R_Q^\pi,u\rangle=0$ for every $\pi\in S_n$. Since span$\{R_Q^\pi:\pi\in S_n\}=V_2^+$, it follows that \(u\perp V_2^+\), and hence $h_u^+=0$.
\end{proof}
Interchanging $w$ and $u$ proves condition~$\ref{itemi}$.
\begin{claim}\label{claim:ii}
If $h_w^-\neq 0$, then $h_u^-=0$.
\end{claim}
\begin{proof}
	For distinct \(a,b,c\), let \(T_{a,b,c}\in V_2^-\) be defined by
	\begin{equation*}
T_{a,b,c}(x,y)=\begin{cases}1, & (x,y)\in\{(a,b),(b,c),(c,a)\},\\-1, & (x,y)\in\{(b,a),(c,b),(a,c)\},\\0, & \text{otherwise}.\end{cases}
	\end{equation*}
We first show that $\{T_{i,j,n}:1\leq i<j\leq n-1\}$ is a basis of \(V_2^-\). They are linearly independent, since the coordinate \((i,j)\) occurs only in \(T_{i,j,n}\). Moreover, for every \(h\in V_2^-\),
\item \textbf{Case 1.} $x<y<n$.

Only \(T_{x,y,n}\) contributes at the coordinate \((x,y)\), and
\(T_{x,y,n}(x,y)=1\). Hence
\begin{equation*}
\left( \sum_{1 \le i < j \le n-1} h(i, j) T_{i,j,n} \right)(x, y) = \sum_{1 \le i < j \le n-1} h(i, j) T_{i,j,n}(x, y)=h(x,y)T_{x,y,n}(x,y)=h(x,y).
\end{equation*}
\item \textbf{Case 2.} $y<x<n$.

Only \(T_{y,x,n}\) contributes at the coordinate \((x,y)\), and \(T_{y,x,n}(x,y)=-1\). Therefore
\begin{equation*}
\left(\sum_{1\le i<j\le n-1} h(i,j)T_{i,j,n}\right)(x,y)=\sum_{1\le i<j\le n-1}h(i,j)T_{i,j,n}(x,y)=h(y,x)T_{y,x,n}(x,y)=h(x,y).
\end{equation*}
\item  \textbf{Case 3.} $x<y=n$.

At the coordinate \((x,n)\), if $i<x$, then $T_{i,x,n}(x,n)=1$, while if $x<j<n$, then $T_{x,j,n}(x,n)=-1$. Hence
\begin{align*}
\left(\sum_{1\le i<j\le n-1}h(i,j)T_{i,j,n}\right)(x,n)&=\sum_{i<x}h(i,x)-\sum_{x<j<n}h(x,j)\\
&=-\sum_{i<x}h(x,i)-\sum_{x<j<n}h(x,j)\\&=-\sum_{{j<n, j\ne x}}h(x,j)=h(x,n).
\end{align*}

\item \textbf{Case 4.} $y<x=n$.

At the coordinate \((n,y)\), if \(i<y\), then \(T_{i,y,n}(n,y)=-1\). If \(y<j<n\), then \(T_{y,j,n}(n,y)=1\). Hence
\begin{align*}
\left(\sum_{1\le i<j\le n-1} h(i,j)T_{i,j,n}\right)(n,y)&=-\sum_{i<y}h(i,y)+\sum_{y<j<n}h(y,j)\\
&=\sum_{i<y}h(y,i)+\sum_{y<j<n}h(y,j)\\&=\sum_{\substack{j<n\\j\ne y}}h(y,j)=-h(y,n)=h(n,y).
\end{align*}
The four cases above show that
\begin{equation*}
h=\sum_{1\le i<j\le n-1}h(i,j)T_{i,j,n}.
\end{equation*}
Thus span$\{T_{i,j,n}:1\leq i<j\leq n-1\}=V_2^-$. For \(h\in V_2^-\), direct calculation gives
\begin{equation*}
\langle h,T_{a,b,c}\rangle=h(a,b)+h(b,c)+h(c,a)-h(b,a)-h(c,b)-h(a,c)=2\bigl(h(a,b)+h(b,c)+h(c,a)\bigr).
\end{equation*}
Since span$\{T_{i,j,n}:1\leq i<j\leq n-1\}=V_2^-$ and \(h_w^-\ne0\) there exist distinct \(a,b,c\) such that  $\langle h_w^-,T_{a,b,c}\rangle\ne0$. Indeed, otherwise \(h_w^-\) would be orthogonal to \(V_2^-\). By skew-symmetry,
\begin{equation*}
\langle h_w^-,T_{a,b,c}\rangle=2\bigl(h_w^-(a,b)+h_w^-(b,c)+h_w^-(c,a)\bigr)\ne0,
\end{equation*}
Let \(S_{\{a,b,c\}}\) be the six permutations of \(\{a,b,c\}\)(i.e. $S_{\{a,b,c\}} = \{\mathrm{id}, (ab), (ac), (bc), (abc), (acb)\}$). At the coordinate \((a,b)\), direct expansion gives
\begin{align*}
&\qquad\sum_{\rho\in S_{\{a,b,c\}}}\operatorname{sgn}(\rho)w^\rho(a,b)\\&=w^{\mathrm{id}}(a, b) + w^{(abc)}(a, b) + w^{(acb)}(a, b) - w^{(ab)}(a, b) - w^{(ac)}(a, b) - w^{(bc)}(a, b)\\&=w(a,b)+w(b,c)+w(c,a)-w(b,a)-w(c,b)-w(a,c),
\end{align*}
where $w^\rho(x, y) = w(\rho^{-1}(x), \rho^{-1}(y))$, $w^{(abc)}(a,b)=w^{(bc)}(b,a)=w(c,a)$ and 
\begin{equation*}
\operatorname{sgn}(\rho) = \begin{cases} 1, & \rho\in\{id,(abc), (acb)\}, \\ -1, & \rho \in\{(ab), (ac), (bc)\}. \end{cases}
\end{equation*}
Then
\begin{align*}
\sum_{\rho\in S_{\{a,b,c\}}}\operatorname{sgn}(\rho)w^\rho(a,b)
&=w(a,b)+w(b,c)+w(c,a)-w(b,a)-w(c,b)-w(a,c)
\\[-15pt]&=\bigl(w(a,b)-w(b,a)\bigr)+\bigl(w(b,c)-w(c,b)\bigr)+\bigl(w(c,a)-w(a,c)\bigr)
\\&=2\bigl(g_w^-(a,b)+h_w^-(a,b)\bigr)+2\bigl(g_w^-(b,c)+h_w^-(b,c)\bigr)+2\bigl(g_w^-(c,a)+h_w^-(c,a)\bigr)
\\&=2\bigl(g_w^-(a,b)+g_w^-(b,c)+g_w^-(c,a)\bigr)+2\bigl(h_w^-(a,b)+h_w^-(b,c)+h_w^-(c,a)\bigr)
\\&=2\bigl(h_w^-(a,b)+h_w^-(b,c)+h_w^-(c,a)\bigr),
\end{align*}
where $g_w^-(a,b)+g_w^-(b,c)+g_w^-(c,a)=(q_a(w)-q_b(w))+(q_b(w)-q_c(w))+(q_c(w)-q_a(w))=0$. The same calculation gives the same value at \((b,c)\) and \((c,a)\), and the opposite value at \((b,a),(c,b),(a,c)\). For \(x\notin\{a,b,c\}\), every \(\rho\in S_{\{a,b,c\}}\) fixes \(x\). Hence
\begin{equation*}
\sum_{\rho\in S_{\{a,b,c\}}}\operatorname{sgn}(\rho)w^\rho(a,x)=w(a,x)+w(b,x)+w(c,x)-w(b,x)-w(c,x)-w(a,x)=0,
\end{equation*}
and similarly for all remaining coordinates. Hence
\begin{equation*}
\sum_{\rho\in S_{\{a,b,c\}}}\operatorname{sgn}(\rho)w^\rho=2\bigl(h_w^-(a,b)+h_w^-(b,c)+h_w^-(c,a)\bigr)T_{a,b,c}.
\end{equation*}
Now let \(\pi\in S_n\). By (\ref{eq:disc-zero-all}), every relabelling of \(w\) is
orthogonal to \(u\), so
\begin{equation*}
\left\langle\left(\sum_{\rho\in S_{\{a,b,c\}}}\operatorname{sgn}(\rho)w^\rho\right)_\pi,u
\right\rangle=2\bigl(h_w^-(a,b)+h_w^-(b,c)+h_w^-(c,a)\bigr)\langle T_{a,b,c}^\pi,u\rangle=0.
\end{equation*}
Since $\langle T_{a,b,c}^\pi,u\rangle=0$ for every $\pi\in S_n$ and $\operatorname{span}\{T_{a,b,c}^{\pi} : \pi \in S_n\} = V_2^-$, it follows that $u\perp V_2^-$. Since \(h_u^-\in V_2^-\), we have $\langle u,h_u^-\rangle=\langle h_u^-,h_u^-\rangle=\|h_u^-\|_2^2=0$, and hence $h_u^-=0$.
\end{proof}
Interchanging \(w\) and \(u\) proves condition~$\ref{itemii}$.
\begin{claim}
	If \(\disc(w,u)=0\), then $\operatorname{tr}(\Gamma_w\Gamma_u)=0$ $\Longleftrightarrow$ $\operatorname{col}(\Gamma_w)\perp\operatorname{col}(\Gamma_u)$ in $\mathbb R^2$.
\end{claim}
\begin{proof}
	By (\ref{eq:channel-decomp}), (\ref{eq:disc-zero-all}), and Claims~\ref{claim:i} and~\ref{claim:ii}, we have $\langle (h_w^+)_\pi, h_u^+ \rangle = \langle (h_w^-)_\pi, h_u^- \rangle=0$. For every $\pi\in S_n$, since symmetric functions are orthogonal to skew-symmetric functions, $\langle w_\pi,u\rangle=	\left\langle(g_w^++g_w^-)_\pi,g_u^++g_u^-\right\rangle=\langle (g_w^+)_\pi,g_u^+\rangle+\langle (g_w^-)_\pi,g_u^-\rangle=0$. Using $\sum_i p_i(w)=\sum_i p_i(u)=0$, we have
	\begin{align*}
\langle (g_w^+)_\pi,g_u^+\rangle
&=\sum_{i\ne j}\bigl(p_{\pi^{-1}(i)}(w)+p_{\pi^{-1}(j)}(w)\bigr)
\bigl(p_i(u)+p_j(u)\bigr)
\\&=2(n-1)\sum_ip_{\pi^{-1}(i)}(w)p_i(u)+2\sum_{i\ne j}p_{\pi^{-1}(i)}(w)p_j(u)
\\&=2(n-1)\sum_ip_{\pi^{-1}(i)}(w)p_i(u)-2\sum_ip_{\pi^{-1}(i)}(w)p_i(u)
\\&=2(n-2)\sum_ip_{\pi^{-1}(i)}(w)p_i(u)
\\&=\langle\alpha_w^\pi,\alpha_u\rangle .
	\end{align*}
	Similarly, since $\sum_i q_i(w)=\sum_i q_i(u)=0$, we have
	\begin{align*}
\langle (g_w^-)_\pi,g_u^-\rangle
&=\sum_{i\ne j}\bigl(q_{\pi^{-1}(i)}(w)-q_{\pi^{-1}(j)}(w)\bigr)\bigl(q_i(u)-q_j(u)\bigr)
\\&=2(n-1)\sum_iq_{\pi^{-1}(i)}(w)q_i(u)-2\sum_{i\ne j}q_{\pi^{-1}(i)}(w)q_j(u)
\\&=2(n-1)\sum_iq_{\pi^{-1}(i)}(w)q_i(u)+2\sum_iq_{\pi^{-1}(i)}(w)q_i(u)
\\&=2n\sum_iq_{\pi^{-1}(i)}(w)q_i(u)
\\&=\langle\beta_w^\pi,\beta_u\rangle.
	\end{align*}
	Therefore $\left\langle(g_w^++g_w^-)_\pi,g_u^++g_u^-\right\rangle=\langle\alpha_w^\pi,\alpha_u\rangle
+\langle\beta_w^\pi,\beta_u\rangle=0$. Applying Lemma~\ref{lem:second-moment}, we obtain
\begin{align*}
\mathbb E_\pi\left(\langle\alpha_w^\pi,\alpha_u\rangle+\langle\beta_w^\pi,\beta_u\rangle
\right)^2
&=\frac{1}{n-1}\Big(\|\alpha_w\|_2^2\|\alpha_u\|_2^2+\|\beta_w\|_2^2\|\beta_u\|_2^2+2\langle\alpha_w,\beta_w\rangle\langle\alpha_u,\beta_u\rangle\Big)
\\&=\frac{1}{n-1}\operatorname{tr}(\Gamma_w\Gamma_u).
\end{align*}
Hence $\operatorname{tr}(\Gamma_w\Gamma_u)=0$. Since
\(\Gamma_w\) is positive semidefinite, let $\lambda_1,\lambda_2\ge0$, we have $\Gamma_w=
\lambda_1 v_1v_1^\mathsf{T}+\lambda_2 v_2v_2^\mathsf{T}$, where \(v_1,v_2\) are orthonormal. Then
\begin{equation*}
\operatorname{tr}(\Gamma_w\Gamma_u)
=\operatorname{tr}\left(\left(\lambda_1v_1v_1^\mathsf{T}+\lambda_2v_2v_2^\mathsf{T}\right)\Gamma_u\right)
=\lambda_1v_1^\mathsf{T}\Gamma_uv_1+\lambda_2v_2^\mathsf{T}\Gamma_uv_2.
\end{equation*}
Since \(\Gamma_u\) is positive semidefinite, if $\operatorname{tr}(\Gamma_w\Gamma_u)=0$, then whenever $\lambda_i>0$, $v_i^\mathsf{T}\Gamma_uv_i=0$. For a positive semidefinite matrix, this gives
\(\Gamma_uv_i=0\). Therefore
\begin{equation*}
\operatorname{col}(\Gamma_w)=\operatorname{span}\{v_i:\lambda_i>0\}\subseteq\ker(\Gamma_u)=
\operatorname{col}(\Gamma_u)^\perp,
\end{equation*}
and hence $\operatorname{col}(\Gamma_w)\perp\operatorname{col}(\Gamma_u)$. 

Conversely, if \(\operatorname{col}(\Gamma_w)\perp\operatorname{col}(\Gamma_u)\), then \(\Gamma_uv_i=0\) whenever \(\lambda_i>0\), and therefore
\begin{equation*}
\operatorname{tr}(\Gamma_w\Gamma_u)=\sum_{i=1}^2\lambda_i v_i^\mathsf{T}\Gamma_uv_i=0.
\end{equation*}
Thus $\operatorname{tr}(\Gamma_w\Gamma_u)=0\Longleftrightarrow\operatorname{col}(\Gamma_w)\perp\operatorname{col}(\Gamma_u)$.
\end{proof}
Conversely, suppose that conditions~$\ref{itemi}$--$\ref{itemiii}$ hold. By conditions~$\ref{itemi}$ and~$\ref{itemii}$, for every $\pi\in S_n$, we have $\langle(h_w^+)_\pi,h_u^+\rangle=\langle(h_w^-)_\pi,h_u^-\rangle=0$. By condition~$\ref{itemiii}$, we have $\mathbb E_\pi\left(\langle\alpha_w^\pi,\alpha_u\rangle+\langle\beta_w^\pi,\beta_u\rangle\right)^2=0$. Hence for every $\pi\in S_n$, we have $\left\langle(g_w^++g_w^-)_\pi,g_u^++g_u^-\right\rangle=0$. Therefore, by (\ref{eq:channel-decomp}), $\langle w_\pi,u\rangle=0$ for every $\pi\in S_n$. Hence $\disc(w,u)=0$.

Finally, let \(\mathcal F\) be pairwise zero-discrepant.
By conditions~$(i)$ and $(ii)$, we have
\begin{equation*}
\sum_{w\in\mathcal F}\mathbf 1_{\{h_w^+\ne0\}}\le1,\qquad \sum_{w\in\mathcal F}\mathbf 1_{\{h_w^-\ne0\}}\le1.
\end{equation*}
By condition~$(iii)$, the nonzero subspaces
\(\operatorname{col}(\Gamma_w)\subseteq\mathbb R^2\) are pairwise
orthogonal, and therefore
\begin{equation*}
\sum_{w\in\mathcal F}\operatorname{rank}(\Gamma_w)=\sum_{w \in \mathcal{F}} \dim \mathrm{col}(\Gamma_w)= \dim \left( \bigoplus_{w \in \mathcal{F}} \mathrm{col}(\Gamma_w) \right)\le2.
\end{equation*}
Consequently,
\begin{equation*}
\sum_{w\in\mathcal F}\left(\operatorname{rank}(\Gamma_w)+\mathbf 1_{\{h_w^+\ne0\}}+\mathbf 1_{\{h_w^-\ne0\}}\right)\le1+1+2=4.
\end{equation*}
For every nonzero centered \(w\in\mathcal F\), we have $\operatorname{rank}(\Gamma_w)+\mathbf 1_{\{h_w^+\ne0\}}+\mathbf 1_{\{h_w^-\ne0\}}\ge1$. Indeed, otherwise \(\operatorname{rank}(\Gamma_w)=0\) and \(h_w^+=h_w^-=0\). Since \(\Gamma_w\) is the Gram matrix of
\(\alpha_w\) and \(\beta_w\), this implies \(\alpha_w=\beta_w=0\), and hence \(g_w^+=g_w^-=0\). Therefore \(w=0\), a contradiction. Thus
\begin{equation*}
|\mathcal F|\le\sum_{w\in\mathcal F}\left(\operatorname{rank}(\Gamma_w)+\mathbf 1_{\{h_w^+\ne0\}}
+\mathbf 1_{\{h_w^-\ne0\}}\right)\le4.
\end{equation*}
This completes the theorem.
\end{proof}
\begin{proof}[\bf Proof of Theorem~\ref{thm:zero-family}]
	Let $a=e_1-e_2$ and $b=e_3-e_4\in\mathbb R^n$. Then for every $x\in[n]$, we have $a_xb_x=0$ and 
	\begin{equation*}
	\sum_x a_x=\sum_x b_x=0.
	\end{equation*}
	For \(x\ne y\), define 
	\begin{equation*}
	f_1(x,y)=a_x+a_y,\qquad f_2(x,y)=a_x-a_y, 
	\end{equation*}
	and
\begin{equation*}
f_3(x,y)=a_xb_y+b_xa_y,\qquad f_4(x,y)=a_xb_y-b_xa_y.
\end{equation*}
Then, $f_1\in V_1^+$ and $f_2\in V_1^-$. For $x\neq y$, we have $f_3(y,x)=f_3(x,y)$ and
\begin{equation*}
\sum_{y\ne x}f_3(x,y)=a_x\sum_{y\ne x}b_y+b_x\sum_{y\ne x}a_y=-a_xb_x-b_xa_x=-2a_xb_x=0.
\end{equation*}
Hence $f_3\in V_2^+$. Similarly, $f_4(y,x)=-f_4(x,y)$ and 
\begin{equation*}
\sum_{y\ne x}f_4(x,y)=a_x\sum_{y\ne x}b_y-b_x\sum_{y\ne x}a_y=-a_xb_x+b_xa_x=0.
\end{equation*}
Thus $f_4\in V_2^-$. Therefore $f_1\in V_1^+$, $f_2\in V_1^-$, $f_3\in V_2^+$ and $f_4\in V_2^-$.
Since these four subspaces are pairwise orthogonal and preserved by relabelling, for $i\ne j$ and every $\pi\in S_n$, we have $\langle f_i^\pi,f_j\rangle=0$. In particular, the four functions are centered and for $i\neq j$, we have $\disc(f_i,f_j)=0$.

We next compute their \(\ell_1\)-norms. Since $a_1=1$, $a_2=-1$ and for $x\notin\{1,2\}$, we have $a_x=0$. Then
\begin{equation*}
\|f_1\|_1=\sum_{y\notin\{1,2\}}\Big(|f_1(1,y)|+|f_1(y,1)|+|f_1(2,y)|+|f_1(y,2)|\Big)=4(n-2).
\end{equation*}
Similarly, for \(f_2\),
\begin{align*}
\|f_2\|_1&=\sum_{y\notin\{1,2\}}\Bigl(|f_2(1,y)|+|f_2(y,1)|+|f_2(2,y)|+|f_2(y,2)|\Bigr)
+|f_2(1,2)|+|f_2(2,1)|\\&=4(n-2)+2+2=4(n-1).
\end{align*}
Finally, \(f_3\) and \(f_4\) are nonzero exactly when one coordinate
lies in \(\{1,2\}\) and the other lies in \(\{3,4\}\). Hence
\begin{equation*}
\|f_3\|_1=\sum_{\substack{x\in\{1,2\}\\y\in\{3,4\}}}\bigl(|f_3(x,y)|+|f_3(y,x)|\bigr)=8,
\end{equation*}
and
\begin{equation*}
\|f_4\|_1=\sum_{\substack{x\in\{1,2\}\\y\in\{3,4\}}}\bigl(|f_4(x,y)|+|f_4(y,x)|\bigr)=8.
\end{equation*}
For $i=1,2,3,4$, define
\begin{equation*}
w_i=\frac{n(n-1)}{\|f_i\|_1}f_i.
\end{equation*}
Then $\|w_i\|_1=n(n-1)$ and $w_i([n])=0$. Moreover, for \(i\ne j\) and every \(\pi\in S_n\),
\begin{equation*}
\langle w_i^\pi,w_j\rangle=\frac{n^2(n-1)^2}{\|f_i\|_1\|f_j\|_1}\langle f_i^\pi,f_j\rangle\\=0.
\end{equation*}
Hence for $1\le i<j\le4$, we have $\disc(w_i,w_j)=0$. This completes the proof.
\end{proof}
\begin{corollary}
	Let \(n\ge4\), and let $\mathcal F=\{w_1,w_2,w_3,w_4\}$ be a pairwise zero-discrepant family of nonzero centered weight functions on \([n]\). Then, after reindexing, $w_1\in V_2^+$ and $w_2\in V_2^-$. Moreover, $w_3,w_4\in V_1^+\oplus V_1^-$. And $\operatorname{rank}(\Gamma_{w_3})=\operatorname{rank}(\Gamma_{w_4})=1$ with $\operatorname{col}(\Gamma_{w_3})\perp\operatorname{col}(\Gamma_{w_4})$.
	Consequently, $\mathbb R^2=\operatorname{col}(\Gamma_{w_3})\oplus\operatorname{col}(\Gamma_{w_4})$. Conversely, any four nonzero centered weight functions satisfying
	these conditions are pairwise zero-discrepant.
\end{corollary}
\begin{proof}[\bf Proof]
	By Theorem~\ref{thm:zero-classification}, we have
	\begin{equation*}
\sum_{w\in\mathcal F}\left(\operatorname{rank}(\Gamma_w)+\mathbf 1_{\{h_w^+\ne0\}}+\mathbf 1_{\{h_w^-\ne0\}}\right)\le4.
	\end{equation*}
	Since every \(w\in\mathcal F\) is nonzero and centered, each
	summand is at least \(1\). As \(|\mathcal F|=4\), equality holds,
	and each member contributes exactly \(1\). Moreover,
	\begin{equation*}
	\sum_{w\in\mathcal F}\mathbf 1_{\{h_w^+\ne0\}}\le1,
	\qquad
	\sum_{w\in\mathcal F}\mathbf 1_{\{h_w^-\ne0\}}\le1,
	\qquad
	\sum_{w\in\mathcal F}\operatorname{rank}(\Gamma_w)\le2.
	\end{equation*}
	Hence after reindexing, assume $h_{w_1}^+\ne0$ and $h_{w_2}^-\ne0$. Then $\operatorname{rank}(\Gamma_{w_1})=0$, $h_{w_1}^-=0$, and $\operatorname{rank}(\Gamma_{w_2})=0$, $h_{w_2}^+=0$. Since \(\operatorname{rank}\Gamma_{w_1}=0\), we have \(\alpha_{w_1}=\beta_{w_1}=0\), and hence $g_{w_1}^+=g_{w_1}^-=0$. Therefore $w_1=h_{w_1}^+\in V_2^+$. Similarly, $w_2=h_{w_2}^-\in V_2^-$. For \(w_3,w_4\), we have $h_{w_3}^\pm=h_{w_4}^\pm=0$, so $w_3,w_4\in V_1^+\oplus V_1^-$. The equality
	\begin{equation*}
	\sum_{w\in\mathcal F}\operatorname{rank}(\Gamma_w)=2,
	\end{equation*}
	then gives $\operatorname{rank}(\Gamma_{w_3})=\operatorname{rank}(\Gamma_{w_4})=1$. Finally, condition~$\ref{itemiii}$ of Theorem~\ref{thm:zero-classification} implies $\operatorname{col}(\Gamma_{w_3})\perp\operatorname{col}(\Gamma_{w_4})$. Since both spaces are one-dimensional subspaces of \(\mathbb R^2\), their orthogonal direct sum is \(\mathbb R^2\). The converse follows directly from conditions~$\ref{itemi}$--$\ref{itemiii}$ of Theorem~\ref{thm:zero-classification}.
\end{proof}
\section*{Declaration on the use of AI}
The authors used ChatGPT 5.6 Pro to assist in discussing proof strategies, checking proofs, and improving exposition.

\end{document}